\documentclass{amsart}
\usepackage{amsmath}
\usepackage{amssymb}
\usepackage{color}
\input xy
\xyoption{all}

\theoremstyle{definition}
\newtheorem{theorem}{Theorem}[section]
\newtheorem{lemma}[theorem]{Lemma}

\theoremstyle{definition}
\newtheorem{definition}[theorem]{Definition}

\newtheorem{c-example}[theorem]{Counter-example}
\newtheorem{Lemma}[theorem]{Lemma}
\newtheorem{corollary}[theorem]{Corollary}
\newtheorem{Prop}[theorem]{Proposition}
\theoremstyle{remark}
\newtheorem{remark}[theorem]{Remark}
\newtheorem{examples}[theorem]{Examples}
\numberwithin{equation}{section}

\usepackage{hyperref}
\usepackage{cancel}
\usepackage[normalem]{ulem}

\newcommand{\Cal}[1]{{\mathcal #1}}
\newcommand{\paral}[1]{\ar@<0.3ex>[#1] \ar@<-0.3ex>[#1]}

\newcommand{\PreOrd}[1]{\mathsf{PreOrd}(\mathbb{#1})}
\newcommand{\ParOrd}[1]{\mathsf{ParOrd}(\mathbb{#1})}
\newcommand{\Equiv}[1]{\mathsf{Eq}(\mathbb{#1})}
\newcommand{\Disc}[1]{\mathsf{Dis}(\mathbb{#1})}

\newcommand{\Stab}{\mathbb{S}}

\newcommand{\Id}[1]{\operatorname{Id}_{#1}}

\DeclareMathOperator{\Sub}{\mathsf{Sub}}
\DeclareMathOperator{\Hom}{hom}
\DeclareMathOperator{\Eq}{Eq}

\newdir{ >}{{}*!/-7pt/@{>}}

\begin{document}

 \title{The stable category of preorders in a pretopos I: general theory}
  
  \author{Francis Borceux}
\address{Universit\'e catholique de Louvain, Institut de Recherche en Math\'ematique et Physique, 1348 Louvain-la-Neuve, Belgium}
\email{francis.borceux@uclouvain.be}

\author{Federico Campanini}
\address{Universit\`a degli Studi di Padova, Dipartimento di Matematica ``Tullio Levi-Civita'', 35121 Padova, Italy}
\email{federico.campanini@unipd.it}

\author{Marino Gran}
\address{Universit\'e catholique de Louvain, Institut de Recherche en Math\'ematique et Physique, 1348 Louvain-la-Neuve, Belgium}
\email{marino.gran@uclouvain.be}

\thanks{
The second author was partially supported by Fondazione Ing. Aldo Gini - Universit\`a di Padova, borsa di studio per l'estero bando anno 2019 and by Ministero dell'Istruzione, dell'Universit\`a e della Ricerca (Progetto di ricerca di rilevante interesse nazionale ``Categories, Algebras: Ring-Theoretical and Homological Approaches (CARTHA)'').\\
This work was also supported by the collaboration project \emph{Fonds d'Appui à l'Internationalisation} ``Coimbra Group'' (2018-2021) funded by the Universit\'e catholique de Louvain.}

\makeatletter
\@namedef{subjclassname@2020}{\textup{2020} Mathematics Subject Classification}
\makeatother


\begin{abstract} 
In a recent article Facchini and Finocchiaro considered a natural pretorsion theory in the category of preordered sets inducing a corresponding stable category. In the present work we propose an alternative construction of the stable category of the category $\PreOrd C$ of internal preorders in any coherent category $\mathbb C$, that enlightens the categorical nature of this notion.
When $\mathbb C$ is a pretopos we prove that the quotient functor from the category of internal preorders to the associated stable category preserves finite coproducts. Furthermore, we identify a wide class of pretoposes, including all $\sigma$-pretoposes and all elementary toposes, with the property that
this functor sends any short $\mathcal Z$-exact sequences in $\PreOrd C$ (where $\mathcal Z$ is a suitable ideal of trivial morphisms) to a short exact sequence in the stable category. These properties will play a fundamental role in proving the universal property of the stable category, that will be the subject of a second article on this topic.
\end{abstract}

\maketitle

\section*{Introduction}
In a recent article \cite{FF} Facchini and Finocchiaro observed that in the category $\mathsf{PreOrd}$ of preordered sets there is a natural \emph{pretorsion theory} $(\mathcal T, \mathcal F) = ( \mathsf{Eq}, \mathsf{ParOrd})$, where $\mathsf{Eq}$ is the ``torsion subcategory'' of equivalence relations and $\mathsf{ParOrd}$ the ``torsion-free'' subcategory of partial orders. Let us write $\mathcal Z = \mathsf{Eq} \cap \mathsf{ParOrd}$ for the full subcategory of $\mathsf{PreOrd}$ whose objects are discrete equivalence relations, and call $\mathcal Z$-trivial a morphism in $\mathsf{PreOrd}$ that factors through an object in $\mathcal Z$. Then the fact that $( \mathsf{Eq}, \mathsf{ParOrd})$ is a pretorsion theory can be expressed as follows:
\begin{enumerate}
\item any morphism $f \colon (X, \tau) \rightarrow (Y, \sigma)$ from an equivalence relation $(X, \tau)$ to a partial order $(Y, \sigma)$ is $\mathcal Z$-trivial; 
\item for any preorder $(A, \rho)$ there is a canonical short $\mathcal Z$-exact sequence
$$
\xymatrix{
(A, \sim_{\rho}) \ar[r]^{\Id A} & (A,\rho)\ar[r]^-\pi & ({A}/{\sim_{\rho}}, \pi(\rho))\\
}
$$
where $\sim_\rho=\rho\cap \rho^\circ$ is the equivalence relation given by the intersection of $\rho$ with its opposite relation $\rho^\circ$, $\pi\colon A \to A/{\sim_\rho}$ is the canonical quotient and $\pi(\rho)$ is the partial order induced on $A/{\sim_\rho}$ by $\rho$.
\end{enumerate}
Note that a short $\mathcal Z$-exact sequence is defined similarly to the usual notion of exact sequence in a pointed category, with the $\mathcal Z$-trivial morphisms having the same role as zero morphisms in the classical pointed setting.
The pair $( \mathsf{Eq}, \mathsf{ParOrd})$ was the first example of the general notion of pretorsion theory introduced in \cite{FF} and thoroughly investigated in \cite{FFG, FFG2}. This notion is a wide generalization of the classical notion of torsion theory for abelian categories due to Dickson \cite{D}, that was later extended to other non-additive contexts by several authors (see for instance \cite{BG, CDT, GJ, JT} and the references therein).

An interesting observation in \cite{FF} is the following: even though $\mathsf{PreOrd}$ is not pointed, it is possible to naturally associate with it a pointed quotient category ${\mathbb S}^*$, called the \emph{stable category}. Via the corresponding quotient functor, two parallel morphisms $f$ and $g$ between (non-empty) preordered sets are identified in ${\mathbb S}^*$ when they coincide on a ``clopen subobject'' of their domain (see Section \ref{Section_clopen}) and they are $\mathcal Z$-trivial (i.e. they factor through a discrete equivalence relation) on the complement of this subobject. This quotient functor $\Sigma$ is shown to send all ``trivial objects'' in $\mathcal Z$ to the zero object of the pointed category ${\mathbb S}^*$. Furthermore, this functor has some interesting properties: it preserves finite coproducts and, more importantly, it sends short $\mathcal Z$-exact sequences to ``genuine'' exact sequences in ${\mathbb S}^*$. 

In the present work, we extend the results in \cite{FF} to the category $\mathsf{PreOrd}(\mathbb C)$ of internal preorders in a coherent category $\mathbb C$ \cite{Elep}, hence including a variety of new examples.
We first extend the construction of the stable category by Facchini and Finocchiaro to the context of coherent categories in Section \ref{Section_stable_category}, where we restrict our attention to the objects of $\mathsf{PreOrd}(\mathbb C)$ with ``global element''. 
We then give a new construction of the stable category as a suitable quotient of a \emph{category of partial morphisms}, that keeps the same objects as the ones in $\mathsf{PreOrd}(\mathbb C)$, thus avoiding the restriction of considering only the objects having a global element (Section \ref{new-def}). Of course, the ``new'' stable category, denoted by $\mathsf{Stab}(\mathbb C)$, coincides with ${\mathbb S}^*$ if we restrict ourselves to the preorders having a global element.

In order to establish the main properties of the quotient functor $\Sigma \colon \mathsf{PreOrd}(\mathbb C) \rightarrow  \mathsf{Stab}(\mathbb C)$ we then assume $\mathbb C$ to be a \emph{pretopos} (from Section \ref{Section_pretoposes} on). In this context we prove that the functor $\Sigma \colon \mathsf{PreOrd}(\mathbb C) \rightarrow  \mathsf{Stab}(\mathbb C)$ preserves finite coproducts (Proposition \ref{coproducts-preserved}). Whereas the existence of $\mathcal Z$-kernels is always guaranteed, to prove that $\mathsf{PreOrd}(\mathbb C)$ has $\mathcal Z$-cokernels we need an additional assumption on the base category $\mathbb C$. In view of Proposition \ref{existence-cokernels} it is natural to assume that $\mathbb C$ is a $\tau$-pretopos, i.e. a pretopos with the property that the transitive closure of any relation on a given object exists. It is well known that any $\sigma$-pretopos has this property \cite{Elep}, however there are interesting examples of $\tau$-pretoposes which are not $\sigma$-pretoposes, such as the category of compact Hausdorff spaces (Example \ref{HComp}). We then prove that, under this assumption, the functor $\Sigma \colon \mathsf{PreOrd}(\mathbb C) \rightarrow  \mathsf{Stab}(\mathbb C)$ sends any short $\mathcal Z$-exact sequence to a short exact sequence in $ \mathsf{Stab} (\mathbb C)$ (Theorem \ref{main-theo}). 

This crucial property will be used in the second article of this series, where a universal property of the stable category $\mathsf{Stab}(\mathbb C)$ will be established. Roughly speaking, this property will express the fact that the stable category provides the ``best possible torsion theory'' one can associate with the pretorsion theory $(\mathsf{Eq}(\mathbb C), \mathsf{ParOrd}(\mathbb C))$ in $\mathsf{PreOrd}(\mathbb C)$.

{{\bf Acknowledgement.} The authors would like to thank Vasileios Aravantinos-Sotiropoulos and the anonymous referee for many useful suggestions they made on a preliminary version of the article.}

\section{Preliminary notions}

Recall that an arrow in a category $\mathbb{C}$ is 
a regular epimorphism when it is the coequalizer of two arrows in $\mathbb{C}$.
A finitely complete category $\mathbb{C}$ is \emph{regular} \cite{Barr} if
\begin{enumerate}
\item any arrow $f \colon X \rightarrow Y$ in $\mathbb{C}$ has a factorization $f= m \circ q$, with $q$ a regular epimorphism and $m$ a monomorphism; 
\item regular epimorphisms are stable under pullbacks in $\mathbb C$.
\end{enumerate}
If $X$ is an object in $\mathbb{C}$, we write $\mathsf{Sub}(X)$ for the category whose objects are the \emph{subobjects} of $X$. As usual, these are defined as equivalence classes of monomorphisms with the same codomain $X$.
A \emph{coherent category} is a regular category in which every $\mathsf{Sub}(X)$ has finite unions and, for any $f \colon X \rightarrow Y$, each pullback functor $f^* \colon \mathsf{Sub}(Y) \rightarrow \mathsf{Sub}(X)$ preserves them \cite{Elep}. 
Any coherent category has an initial object, denoted by $0$, which is \emph{strict}: any morphism with codomain $0$ is an isomorphism.
The initial object $0$ is the domain of the smallest subobject of the terminal object $1$.
In a coherent category $\mathbb C$ the distributive law
$$
A\cap(B\cup C)\cong (A\cap B)\cup (A\cap C)
$$
holds for any $A,B,C\in \mathsf{Sub}(X)$ \cite[Lemma~1.4.2]{Elep}.
Moreover, any commutative square in $\mathbb C$ of the form
$$
\xymatrix{
A\cap B \ar[r]\ar[d] & B \ar[d]\\
A\ar[r] & A\cup B \\
}
$$
is both a pullback and a pushout. This implies in particular that if $A$ and $B$ are disjoint subobjects of $X$, i.e. $A\cap B\cong 0$, then $A\cup B$ is the coproduct $A\coprod B$ of $A$ and $B$ in $\mathbb{C}$.
If we denote by $p_1 \colon X\times Y\rightarrow X$ the first projection of the product $X \times Y$, the pullback functor ${p_1^*}: \mathsf{Sub}(X)\to \mathsf{Sub}(X\times Y)$ preserves unions, therefore we have that $$(A\cup B) \times Y \cong (A\times Y) \cup (B\times Y)$$
for any subobjects $ x \colon A \rightarrow X$ and $y \colon B \rightarrow X$ in $\mathsf{Sub}(X)$, since both the squares 
$$
\xymatrix{
A \times Y \ar[r] \ar[d]_{x \times 1_Y} & A \ar[d]^x \ar@{}[drr]|{\mbox{and}}& & B \times Y \ar[r] \ar[d]_{y \times 1_Y} & B \ar[d]^y \\
X \times Y \ar[r]_-{p_1} & X & & X \times Y \ar[r]_-{p_1} & X \\
}
$$
are pullbacks.

\begin{examples}
Any pretopos, and in particular any topos, is a coherent category. The category $\mathsf{CHaus}$ of compact Hausdorff spaces and continuous maps is (coherent and is) a pretopos. Its full subcategory $\mathsf{Stone}$ of Stone spaces (i.e. totally disconnected compact Hausdorff spaces) is coherent as well, but is not a pretopos, since it is not exact in the sense of \cite{Barr}.
Any distributive lattice, seen as a preorder, is a coherent category. In general, given a coherent category $\mathbb C$, any functor category $[\mathbb D, \mathbb C]$ is again coherent, as well as every localization of such a functor category (see \cite{Elep, MR} for more examples).
\end{examples}


A subobject $A$ of an object $X$ in a coherent category is \emph{complemented} if there exists another subobject of $X$, denoted by $A^c$, such that $A \cap A^c = 0$ and $A \cup A^c = X$. The complement $A^c$ of a subobject $A$ is unique, if it exists. Moreover, we have
$$
X\times X= (A \coprod A^c) \times (A \coprod A^c)  \cong(A\times A)\coprod(A\times A^c)\coprod(A^c\times A)\coprod(A^c\times A^c)
$$
as it immediately follows from the properties of a coherent category recalled above. 
\vspace{3mm}

\begin{remark}\label{complementi}
Notice that if $X$ is an object of a coherent category $\mathbb C$, then $\Sub (X)$ is a bounded distributive lattice, with join and meet operations given by unions and intersections respectively (cf. \cite[Section~A1.4]{Elep}). In particular, if $B,C \in \Sub(X)$ are complemented subobjects of $X$, then $B\cup C$ and $B\cap C$ are still complemented subobjects of $X$ and the following isomorphisms hold:
$$
(B\cup C)^c\cong B^c\cap C^c \qquad (B\cap C)^c\cong B^c\cup C^c.
$$
Moreover, it is clear that for any morphism $f\colon Y\to X$ in $\mathbb C$, the pullback functor $f^*\colon \Sub(X)\to \Sub(Y)$ preserves all existing complements.
\end{remark}

\noindent {\bf Convention.} In this section $\mathbb C$ will always denote a \emph{coherent category}.
\vspace{3mm}

Let us now recall the definition of the category $\mathsf{PreOrd}(\mathbb C)$ of (internal) preorders in $\mathbb C$. An object $(A, \rho)$ in $\mathsf{PreOrd}(\mathbb C)$ is a relation $\langle r_1,r_2 \rangle \colon \rho \rightarrow A \times A$ on $A$, i.e. a subobject of $A \times A$, that is \emph{reflexive}, i.e. it contains the ``discrete relation'' $\langle 1_A,1_A \rangle \colon A \rightarrow A \times A$ on $A$, usually denoted by $\Delta_A$ -
and \emph{transitive}: there is a
morphism $\tau \colon \rho \times_A \rho \rightarrow \rho$ such that $r_1 \tau =  r_1 p_1$ and $r_2 \tau =  r_2 p_2$, where $\rho \times_A \rho$ is the ``object part'' of the pullback
$$
\xymatrix{ \rho \times_A \rho \ar[r]^-{p_2} \ar[d]_-{p_1}& \rho \ar[d]^{r_1} \\
\rho \ar[r]_{r_2} & A.
}
$$ 

A morphism $(A, \rho) \rightarrow  (B, \sigma)$ in the category $\mathsf{PreOrd}(\mathbb C)$ of preorders in $\mathbb C$ is a pair of morphisms $(f,\hat{f})$ in $\mathbb C$ making the following diagram commute
\begin{equation}\label{morphism-1}
\xymatrix{  \rho \ar@<.5ex>[d]^{r_2} \ar@<-.5ex>[d]_{r_1}  \ar[r]^{\hat{f}}  & \sigma \ar@<.5ex>[d]^{s_2} \ar@<-.5ex>[d]_{s_1}  \\
 A \ar[r]_{f} & {B,} &
}
\end{equation}
in the sense that $f r_1= s_1 \hat{f} $ and $f r_2= s_2 \hat{f}$. For ease of notation, we shall often write $f$ instead of $(f,\hat{f})$ for a morphism in $\PreOrd C$ (remark that $\hat{f}$ is uniquely determined by $f$, when it exists). Note that the category $\mathsf{PreOrd} (\mathbb{C})$ is generally not regular when the category $\mathbb{C}$ is coherent: indeed, it is well-known that regular epimorphisms are not pullback stable even in $\mathsf{PreOrd} (\mathsf{Set})$ ({see \cite[Section~2]{JS}, or \cite[Example 2.4]{Lorengian}}, for instance).

Recall that an \emph{equivalence relation} in $\mathbb C$ is a preorder $(A, \rho)$ as above which is also \emph{symmetric}: there is an arrow $s \colon \rho \rightarrow \rho$ such that $r_1 s = r_2$  and $r_2 s = r_1$. Equivalently, one can ask that the opposite relation $\rho^\circ$ of $\rho$ is isomorphic to $\rho$: $\rho^\circ = \rho$. We write $\Equiv C$ for the full subcategory of $\PreOrd C$ whose objects are the equivalence relations in $\mathbb C$. A \emph{partial order} in $\mathbb C$ is a preorder $(A, \rho)$ having the additional property that $\rho \cap \rho^\circ = \Delta_A$, where $\Delta_A$ is the discrete equivalence relation on $A$ (that is, the equality relation on $A$). $\ParOrd C$ will denote the full subcategory of $\PreOrd C$ whose objects are the partial orders in $\mathbb C$.
 {In $\PreOrd C$ there is a class $\mathcal Z:=\ParOrd C \cap \Equiv C$ of \emph{trivial objects}: the \emph{discrete} equivalence relations or, in other words, the equality relations}. As in \cite{FF, FFG2} we shall say that a morphism \eqref{morphism-1} in $\PreOrd C$ is \emph{trivial} if it factors through a discrete equivalence relation.
By the universal property of the kernel pair $(\Eq(f), f_1, f_2)$ of $f$, a morphism \eqref{morphism-1} in $\PreOrd C$ is trivial if and only if there is a (unique) morphism $i \colon  \rho \rightarrow \Eq(f)$ making the left-hand side of the diagram
$$
\xymatrix@=30pt{  & \rho \ar@<.5ex>[d]^{r_2} \ar@{.>}[ld]_{i} \ar@<-.5ex>[d]_{r_1}  \ar[r]^{\hat{f}}  & \sigma \ar@<.5ex>[d]^{s_2} \ar@<-.5ex>[d]_{s_1}  \\
\Eq(f) \ar@<.5ex>[r]^{f_1} \ar@<-.5ex>[r]_{f_2}& A \ar[r]_{f} & {B.} &
}
$$
{commute in the sense that $f_k i=r_k$ $(k=1,2)$.} This class of \emph{trivial morphisms}, also called ${\mathcal Z}$-trivial, {form an {\em ideal of morphisms} in the sense of Ehresmann~\cite{Ehr}. This means that if $(f,\hat{f})$ is a trivial morphism, then every composite of the form $(h,\hat{h})(f,\hat{f})(g,\hat{g})$ is again trivial (whenever it is defined). These morphisms} play the role of the ``zero morphisms'' in this non-pointed context, leading to a natural notion of $\Cal Z$-kernel. Indeed, if  $(f, \hat{f}) \colon (A, \rho) \rightarrow  (B, \sigma)$ is a morphism in $\PreOrd C$, a morphism $(\varepsilon,\hat{\varepsilon}) \colon (K, \tau) \rightarrow  (A, \rho) $ is a \emph{$\Cal Z$-kernel} of $(f, \hat{f})$ if the composite $(f, \hat{f}) (\varepsilon,\hat{\varepsilon})$ is a $\Cal Z$-trivial morphism in $\PreOrd C$ and, moreover, whenever $(\lambda, \hat{\lambda}) \colon (D, \psi) \rightarrow (A, \rho)$ is a morphism in $\PreOrd C$ such that $(f, \hat{f})  (\lambda, \hat{\lambda})$ is $\Cal Z$-trivial, then $(\lambda, \hat{\lambda})$ factors uniquely through $(\varepsilon, \hat{\varepsilon})$. The notion of $\Cal Z$-cokernel is defined dually. Note that the notions of kernels, cokernels and short exact sequence with respect to an ideal of morphisms played an important role in the work of Lavendhomme \cite{Lav} and Grandis \cite{Gra1, Gra2}. More recently, this approach has also led to a unification of some results in pointed and non-pointed categorical algebra \cite{GJU}.

When $\mathbb C$ is finitely complete, any morphism in 
$\PreOrd C$ has a 
$\Cal Z$-kernel:
\begin{Lemma}\label{Prop_prekernel}
Let $(f, \hat{f}) \colon (A, \rho) \rightarrow  (B, \sigma)$ be a morphism in $\PreOrd C$. Then
its $\mathcal Z$-kernel exists and is given by the inclusion
$$
 (1_A, \hat{\varepsilon})\colon (A, \Eq(f)\cap \rho)\to (A,\rho),
$$
of $\Eq(f)\cap \rho$ in  $\rho$.
\end{Lemma}
\begin{proof}
 Since $\Eq(f)\cap \rho\rightrightarrows A$ factors through $\Eq(f)$, the morphism $(f, \hat{f}) (1_A, \hat{\varepsilon})$ is trivial. Moreover, if $(\lambda, \hat{\lambda}) \colon (D,\psi) \rightarrow  (A, \rho)$ is another morphism such that $(f, \hat{f}) (\lambda, \hat{\lambda})$ is trivial, then $\psi \rightrightarrows D \to A$ factors through $\Eq(f)$ as well:
 $$
\xymatrix@=15pt{
& &\Eq(f)\cap \rho \ar@/_/[ddll]\ar[rr]^{\hat{\varepsilon}}\paral{dd} & & \rho \ar[rrr]^{\hat{f}} \paral{dd}  &&& \sigma \paral{dd} \\
& & & \psi \paral{dd}\ar[ur]^{\hat{\lambda}} \ar@/_/[dlll] \ar@{.>}[ul]^{\hat{u}}& & & & \\
\Eq(f) \paral{rr} & & A \ar@{=}[rr] & & A\ar[rrr]^f & & & B. \\
& & & D\ar[ur]^{\lambda} \ar@{.>}[ul]^{u} & & & \\
}
$$
Accordingly, there is a unique morphism $(u, \hat{u}) \colon (D,\psi)\to (A, \Eq(f)\cap \rho)$ such that $(\lambda, \hat{\lambda}) =(1_A, \hat{\varepsilon}) (u, \hat{u})$.

\end{proof}
The existence of $\mathcal Z$-cokernels requires stronger assumptions on the base category $\mathbb C$, and will be investigated more thoroughly in Section \ref{Short exact sequences}.
 The following observation will be useful later on:
 
\begin{Lemma}\label{epi-mono-trivial}
Given two composable morphisms $(f,\hat{f})\colon (A,\rho) \to (B,\sigma)$ and \\ $(g, \hat{g})\colon (B, \sigma)\to (C,\tau)$ in $\PreOrd C$ such that $(g,\hat{g})(f,\hat{f}) \colon (A,\rho) \to (C,\tau)$ is trivial, then
\begin{itemize}
\item
if $g\colon A \to B$ is a monomorphism in $\mathbb{C}$, then $(f,\hat{f})$ is trivial;
\item
if $\hat{f}$ is an epimorphism in $\mathbb{C}$, then $(g,\hat{g})$ is trivial. 
\end{itemize}
\end{Lemma}

\section{Clopen subobjects}\label{Section_clopen}

Let $(A,\rho)$ be an object in $\PreOrd {\mathbb C}$. For a subobject $j \colon B \rightarrow A$ in $\mathbb{C}$, we denote by $\rho_B$ the internal preorder on $B$ obtained by ``restricting" the relation $\rho$ to $B$. This is expressed by the fact that $\rho_B$ is defined by the following pullback:
$$
\xymatrix{
\rho_B \ar[r] \ar[d] & \rho \ar[d]\\
B\times B \ar[r]_{j \times j} & A \times A.\\
}
$$

By an \emph{open} subobject $B$ of $(A,\rho)$, we mean (with a slight abuse of notation) a subobject $(B,\rho_B)$ of $(A,\rho)$ where $B$ is a complemented subobject of $A$ and the square
\begin{equation}\label{open}
\xymatrix{
0 \ar[r] \ar[d] & \rho \ar[d]\\
B^c\times B \ar[r] & A\times A, \\
}
\end{equation}
where the lower horizontal arrow is the product subobject, is a pullback, i.e. $(B^c\times B) \cap \rho=0$. 

This terminology is justified by the following observations. 
When $\mathbb C$ is the category $\mathsf{Set}$ of sets, there is a category isomorphism between the category $\mathsf{Alex}$ of Alexandroff-discrete spaces (i.e. the topological spaces having the property that the intersection of any family of open sets is open) and the category $\mathsf{PreOrd}(\mathsf{Set})$ of preordered sets. This category isomorphism associates, with any preordered set $(A, \rho)$, the topological space $(A, \tau_{\rho})$ whose open sets $B$ are the subsets $B$ of $A$ satisfying the following property: if $a\in A$, $b \in B$ and $(a,b) \in \rho$, then $a \in B$.
In a coherent category this property is expressed by the requirement that the diagram \eqref{open} is a pullback.

We say that $B$ is a \emph{clopen} subobject of $(A,\rho)$ when $B$ is a complemented subobject of $A$ and both $(B,\rho_B)$ and $(B^c,\rho_{B^c})$ are open subobjects of $(A,\rho)$, that is, both the commutative diagrams
$$
\xymatrix{
0 \ar[r] \ar[d] & \rho \ar[d] \ar@{}[drr]|{\mbox{and}}& & 0 \ar[r] \ar[d] & \rho \ar[d]\\
B^c\times B \ar[r] & A\times A & & B\times B^c \ar[r] & A\times A \\
}
$$
are pullbacks. As one might expect, finite unions, finite intersections, complements and inverse images of clopen subobjects are still clopen. More precisely, we have the following result:

\begin{lemma}\label{lemma1_clopen}
Let $\mathbb C$ be a coherent category, $(A,\rho)$ an object of $\PreOrd C$ and $B,C$ subobjects of $(A,\rho)$. Then:
\begin{enumerate}
\item
if $B$ is clopen in $(A,\rho)$, then so is $B^c$;
\item
if $B$ and $C$ are open (respectively, clopen) subobjects of $(A,\rho)$, then $B\cup C$ and $B\cap C$ are open (respectively, clopen) subobjects of $(A,\rho)$;
\item
let $(f,\hat{f})\colon(X,\sigma)\to(A,\rho)$ be a morphism in $\PreOrd C$. If $B$ is an open (respectively, clopen) subobject of $(A,\rho)$, then the inverse image $f^*(B)$ is an open (respectively, clopen) subobject of $(X,\sigma)$.
\end{enumerate}
\end{lemma}

\begin{proof}
$(1)$ Obvious from the definition of clopen subobject.

$(2)$ It suffices to prove that $B\cup C$ is an open subobject of $(A,\rho)$. This immediately follows from the fact that, under our assumptions, $$\rho\cap ((B\cup C)^c\times (B\cup C))=\rho\cap (((B^c\cap C^c)\times B)\cup ((B^c\cap C^c)\times C))$$ is a subobject of $\rho\cap ((B^c\times B)\cup (C^c\times C))=0$.

$(3)$ First observe that there is a canonical inclusion of $\sigma$ in $f^*(\rho)$ that is induced by the universal property of the pullback defining $f^*(\rho)$:
$$
\xymatrix{
\sigma \ar@/_/@<0.3ex>[ddr] \ar@/_/@<-0.3ex>[ddr] \ar@/^/[drr]^{\bar{f}} \ar@{.>}[dr]\\
 & f^*(\rho) \paral{d} \ar[r]& \rho \paral{d}\\
 & X \ar[r]^f   & A.
}
$$
If $B$ is an open subobject of $(A,\rho)$, then $f^*(B)$ is open in $(X,f^*(\rho))$. Indeed,
$$
(f^*(B)^c \times  f^*(B)) \cap  f^*(\rho) = (f \times f)^*( (B^c\times B) \cap  \rho )= (f \times f)^*(0)  =   0,
$$
where the first equality follows from the fact that $f^*$ preserves complements, and the last one from the property that the initial object $0$ is strict.
The fact that $\sigma \le f^*(\rho)$ then easily implies that 
$f^*(B)$ is also an open subobject of $(X,\sigma)$.
\end{proof}

\begin{corollary}\label{corollary1_clopen}
Let $(A,\rho)$ be an object of $\PreOrd C$. Take $X\in \mathsf{Sub}(A)$, $B\in \mathsf{Sub}(X)$ and assume that $B$ is open (respectively, clopen) in $(A,\rho)$. Then $B$ is open (respectively, clopen) in $(X,\rho_X)$.
\end{corollary}

\begin{proof}
Apply Lemma~\ref{lemma1_clopen} to the monomorphism $(X, \rho_X)\to (A,\rho)$.
\end{proof}

\begin{lemma}\label{lemma2_clopen}
Let $B$ be a clopen subobject of $(A,\rho)$. Then
$
\rho_{_A} \cong \rho_{_B} \coprod \rho_{_{B^c}}.
$
\end{lemma}

\begin{proof} For a clopen subobject $B$ of $A$ we have the situation described by the following diagram, where all arrows are monomorphisms and all squares are pullbacks:
$$
\xymatrix@=30pt@!0{
 & & & \rho_{_{B^c}}\ar[dd]\ar[dl] & \\
0 \ar[rr]\ar[dd] & & \rho_{_A} \ar[dd]& & 0 \ar[ll]\ar[dd] \\
 & \rho_{_B} \ar[dd]\ar[ur] & & B^c\times B^c\ar[dl] & \\
B\times B^c \ar[rr] & & A\times A & & B^c\times B \ar[ll] \\
& B\times B \ar[ur]& & & \\
}
$$
As we recalled in the first section, we know that $$A\times A\cong (B\times B)\coprod (B\times B^c) \coprod (B^c\times B) \coprod (B^c\times B^c)$$ so that, by the coherence of $\mathbb C$, it follows that $\rho_{_A}\cong\rho_{_B} \coprod \rho_{_{B^c}}$.
\end{proof}

\begin{corollary}\label{remark_disjoint_clopen}
If $B$ and $C$ are disjoint clopen subobjects of $(A,\rho)$, then $\rho_{B\coprod C}\cong\rho_B \coprod \rho_C$.
\end{corollary}
\begin{proof}
This follows from Corollary~\ref{corollary1_clopen} and Lemma~\ref{lemma2_clopen}.
\end{proof}

\section{The Facchini-Finocchiaro approach to the stable category}\label{Section_stable_category}

Let us begin by briefly explaining, in the context of a coherent category $\mathbb C$, the construction of the \emph{stable category} of $\PreOrd C$ given in \cite{FF} in the case $\mathbb C$ is the category $\mathsf{Set}$ of sets. We omit some details, since these results will follow from the ones in the next section. 

When $f\colon A\to A'$ is a morphism in $\mathbb{C}$ and $B$ a subobject of $A$, with representing monomorphism $i: B\rightarrow A$, we shall often write $f|_B$ for the composite $f i$, that is the restriction of $f$ to $B$.

\medskip

For every pair of objects $(A, \rho)$ and $(A', \rho')$ in $\PreOrd C$, let $R_{A,A'}$ be the relation on the set $\Hom((A,\rho),(A',\rho'))$ defined, for every $f,g: (A, \rho)\to (A', \rho')$, by $fR_{A,A'}g$ if there exists a clopen subobject $B$ of $(A, \rho)$ such that $f|_{B}=g|_{B}$ and 
$f|_{B^c}$ and $g|_{B^c}$ are two trivial morphisms.

\begin{Prop}
Let $\mathbb C$ be a coherent category. Then the assignment associating the relation $R_{A,A'}$ with any $((A, \rho),(A', \rho')) \in \PreOrd C \times \PreOrd C$ is a congruence (in the sense of \cite{MacLane}, p. 51) on the category $\PreOrd C$.
\end{Prop}

\begin{proof}
The relations $R_{A,A'}$ are clearly reflexive and symmetric. Let us prove that they are also transitive. Fix $f,g,h: (A, \rho)\to (A', \rho')$ with $fR_{A,A'}g$ and $gR_{A,A'}h$. By definition, there exist two clopen subobjects $B,C$ of $(A,\rho)$ such that {$f|_{B}=g|_{B}$}, {$g|_{C}=h|_{C}$}, while
$f|_{B^c},g|_{B^c},g|_{C^c}$ and $h|_{C^c}$ are four trivial morphisms. 
By Lemma~\ref{lemma1_clopen} we already know that $B^c\cup C^c$ is a clopen subobject of $(A,\rho_A)$, and one clearly has that
 $f|_{(B^c \cup C^c)^c}=h|_{(B^c\cup C^c)^c}$, since $(B^c \cup C^c)^c \cong B \cap C$.

We are going to prove that both $f|_{B^c\cup C^c}$ and $h|_{B^c\cup C^c}$ are trivial morphisms. We then write $B^c \cup C^c \cong B^c \coprod C_0$, where $C_0:= B \cap (B^c\cup C^c)$ is a complement of $B^c$ in $B^c\cup C^c$. Notice that $C_0$ is a clopen subobject of $(A,\rho)$ (by Lemma~\ref{lemma1_clopen}) and therefore Corollary~\ref{corollary1_clopen} ensures that both $B^c$ and $C_0$ are clopen subobjects of $(B^c\cup C^c, \rho_{B^c \cup C^c})$.  By Corollary~\ref{remark_disjoint_clopen}, we also have that $\rho_{B^c \cup C^c}\cong\rho_{B^c} \coprod\rho_{C_0}$. Moreover, the morphism $h|_{C_0}$ is trivial, since it is the composite of the trivial morphism $h|_{C^c}$ with the monomorphism $C_0\to C^c$. The situation is described by the following diagram (explained below):
$$
\xymatrix@=25pt{
	 & \rho_{B^c} \ar@/_/[dl]\ar[rd] \paral{d}& & \\
\Eq(f|_{B^c})\paral{r} \ar[d]& {B^c\, \,  } \ar[dr] & \rho \paral{d}\paral{rr} && \rho' \paral{d}\\
\Eq(f|_{B^c \coprod C_0) }& \rho_{C_0} \ar@/_/[ld] \paral{d}\ar[ur] & A \paral{rr}^f_g & &A'\\
\Eq(g|_{C_0})=\Eq(f|_{C_0})\ar[u]\paral{r} & C_0 \ar[ur]& & \\
}
$$
Since $f|_{B^c}$ is a trivial morphism, $\rho_{B^c}$ factors through $\Eq(f|_{B^c})$. Similarly, $\rho_{C_0}$ factors through $\Eq(g|_{C_0})$. Moreover, $\Eq(g|_{C_0})=\Eq(f|_{C_0})$ because $C_0$ is (isomorphic to) a subobject of $B^c$ and $f|_{B^c}=g|_{B^c}$. The two morphisms $\rho_{B^c} \to \Eq(f|_{B^c \coprod C_0})$ and $\rho_{C_0}\to\Eq(f|_{B^c \coprod C_0})$ and the universal property of the coproduct gives a unique (dotted) factorization as in the diagram
$$
\xymatrix{
 &{ \rho_{B^c \cup C^c}\cong\rho_{B^c} \coprod\rho_{C_0} \,\,} \ar[r]\paral{d}\ar@/_/@{.>}[dl] & \rho \paral{d}\ar[rr] && \rho' \paral{d}\\
\Eq(f|_{B^c \coprod C_0})\paral{r} & B^c \coprod C_0 \ar[r] & A \ar[rr]^f & &A'\\
}
$$
showing that $f|_{B^c \cup C^c}$ is a trivial morphism. Similarly, $h|_{B^c \cup C^c}$ is a trivial morphism.
This proves that $R_{A,A'}$ is an equivalence relation on $\Hom((A,\rho),(A',\rho'))$ for each pair of objects $(A,\rho)$ and $(A',\rho')$ in $\PreOrd C$.

It remains to prove the ``compatibility" of the equivalence relations $R_{A,A'}$ with the composition. We omit the rather long verification of this fact since, as already mentioned, an alternative approach will be presented in the next section.
\end{proof}

The quotient category $\PreOrd C/R$ has the same objects as $\PreOrd C$, and is denoted by $\Stab$. It was called the \emph{stable category} by Facchini and Finocchiaro in \cite{FF}. For the sake of precision, since we shall be working with another stable category, we shall call it the FF-stable category. Given a morphism $f:(A,\rho)\to (B,\sigma)$ in $\PreOrd C$, the corresponding morphism in $$\Hom_\Stab ((A,\rho),(B,\sigma))=\Hom_{\PreOrd C}((A,\rho),(B,\sigma))/R_{A,B}$$ is denoted by $\underline{f} \colon (A,\rho)\to (B,\sigma)$. The functor
$$
{\mathcal S} \colon \PreOrd C \longrightarrow  \Stab
$$
is the canonical quotient that is the identity on objects and such that ${\mathcal S}(f)=\underline{f}$ for every morphism $f$ in $\PreOrd C$.

In order to have a ``pointed version'' of the stable category $\Stab$ one can then proceed as in \cite{FF}.
It is clear from the definition of the FF-stable category that if $(Z,\Delta_Z)$ is a trivial object in $\PreOrd C$ and $(A,\rho)$ is any other object in $ \PreOrd C$, then both $\Hom_\Stab((Z,\Delta_Z),(A,\rho))$ and $\Hom_\Stab((A,\rho), (Z,\Delta_Z))$ have \emph{at most} one element. In order to obtain a pointed category $\Stab^*$ out of $\Stab$, one then considers the full subcategory $\PreOrd C^*$ of $\PreOrd C$ whose objects $(A,\rho)$ admit a ``global element'' $(1,\Delta_1)\to (A,\rho)$, so that the unique arrow $(A,\rho) \to (1,\Delta_1)$ is a split epimorphism (notice that $(A,\rho)$ has a global element if and only if $A$ does).
Then, the canonical quotient functor ${\mathcal S} \colon \PreOrd C\to \Stab$
induces a functor
\begin{equation}\label{quotient-category}
{\mathcal S}^*\colon \PreOrd C^* \longrightarrow \Stab^*
\end{equation}
where $\Stab^*$ denotes the full image of $\PreOrd C^*$ under $\mathcal S$.
The following Proposition shows that $\Stab^*$ is a pointed category having the following property: the objects of $\PreOrd C^*$ that become isomorphic to the zero object in $\Stab^*$ are precisely the trivial ones.

\begin{Prop}\label{trivial objects}
The category $\Stab^*$ is a pointed category, whose zero object is the terminal object $(1,\Delta_1)$ of $\PreOrd C$. Moreover, the following conditions hold:
\begin{enumerate}
\item
if $(Z, \zeta)$ is an object in $\PreOrd C^*$, then $(Z,\zeta)\cong (1,\Delta_1)$ in $\Stab^*$ if and only if $(Z,\zeta) \in \Disc C$;
\item
if $f\colon (A,\rho)\to (B, \sigma)$ is a morphism in $\PreOrd C^*$, then $\underline{f}$ is the zero morphism in $\Stab^*$ if and only if $f$ is a trivial morphism.
\end{enumerate}
\end{Prop}

\begin{proof}
It is clear that the terminal object $(1,\Delta_1)$ of $\PreOrd C$, which is subinitial and terminal in $\Stab$, becomes the zero object in $\Stab^*$.
Furthermore, for any trivial object $(Z,\Delta_Z)$ in $\PreOrd C^*$, we have that $(Z,\Delta_Z)$ and $(1,\Delta_1)$ are isomorphic in $\Stab$, because both $\Hom_\Stab((1,\Delta_1),(1,\Delta_1))$ and $\Hom_\Stab((Z,\Delta_Z),(Z,\Delta_Z))$ consist of one element. It follows that every trivial object of $\PreOrd C^*$ becomes (isomorphic to) the zero object in $\Stab^*$.
On the other hand, let $(Z, \zeta)$ be an object in $\PreOrd C^*$ and assume that $(Z,\zeta)\cong (1,\Delta_1)$ in $\Stab^*$. By assumption there are morphisms $\varphi\colon(Z,\zeta)\to (1, \Delta_1)$ and $\psi\colon (1,\Delta_1)\to (Z,\zeta)$ such that $\underline{\psi\varphi}=\underline{\Id Z}$ and $\underline{\varphi\psi}=\underline{\Id 1}$ in $\Stab^*$. The first equality implies that there exists a clopen subobject $C$ of $(Z,\zeta)$ such that $\psi\varphi|_C=i$, and both $\psi\varphi|_{C^c}$ and $j$ are trivial morphisms, where $i\colon C\to Z$ and $j:C^c\to Z$ are the monomorphisms representing the two subobjects. Since $j$ is trivial, it is immediate to check that $(C^c,\zeta_{C^c})\in \Disc C$. Moreover, $\psi\varphi|_C=i$ means that $i$ factors through $(1,\Delta_1)$, hence also $i$ is trivial. It follows that both $(C,\zeta_C)$ and $(C^c,\zeta_{C^c})$ are in $\Disc C$ and, by Lemma~\ref{lemma2_clopen}, $$\zeta=\zeta_C\coprod \zeta_{C^c}=\Delta_C\coprod\Delta_{C^c}=\Delta_Z.$$

Finally, assume that $\underline{f}\colon (A,\rho)\to (B, \sigma)$ is a zero morphism in $\Stab^*$. It means that there exist a trivial object $(Z,\Delta_Z)$ and morphisms $g\colon(A, \rho)\to (Z,\Delta_Z)$ and $h\colon (Z, \Delta_Z)\to (B, \sigma)$ such that $\underline{f}=\underline{h}\underline{g}$. Thus, there exists a clopen subobject $C$ of $(A,\rho_A)$ such that both $f|_{C^c}$ and $hg|_{C^c}$ are trivial and $f|_C=hg|_C$, so that, in particular, also $f|_C$ is trivial. It follows that $\rho_C$ and $\rho_{C^c}$ factor through $\Eq(f|_C)$ and $\Eq(f|_{C^c})$, respectively, hence $\rho_A\cong\rho_C\coprod \rho_{C^c}$ factors through $\Eq(f)$.
Therefore $f$ is a trivial morphism. The other implication is clear.
\end{proof}

 \section{The new definition of the stable category}\label{new-def}

In this section we are going to define an alternative stable category of the category $\mathsf{PreOrd}(\mathbb C)$ of internal preorders in a coherent category $\mathbb C$. In this alternative approach we shall consider \emph{all} the objects of $\mathsf{PreOrd}(\mathbb C)$ (not only the ones  having a ``global element''), and by suitably identifying pairs of parallel morphisms we shall again obtain a \emph{pointed} stable category. 
 When restricted to the objects with a global element, the constructions of the two stable categories will be shown to coincide (Proposition \ref{Coincidence}).

 The first step is to define the \emph{category $\mathsf{PaPreOrd}(\mathbb C)$ of partial morphisms} in $\mathsf{PreOrd}(\mathbb C)$.
 The objects of $\mathsf{PaPreOrd}(\mathbb C)$ are the internal preorders $(A, \rho)$ in $\mathbb C$. A morphism $f \colon (A,\rho) \rightarrow (B,\sigma)$ in the category $\mathsf{PaPreOrd}(\mathbb C)$ is a pair $(\alpha, f')$ displayed as 
$$\xymatrix{& (A',\rho') \ar@{>->}[dl]_{\alpha} \ar[dr]^{f'} & \\ (A, \rho) \ar@{.>}[rr]_f& & (B,\sigma) }
 $$
where $(A', \rho')$ is an internal preorder, $f'$ is a morphism in $\mathsf{PreOrd}(\mathbb C)$, and $\alpha \colon (A',\rho') \rightarrow (A,\rho)$ is a clopen subobject.
The composite $g \circ f$ in $\mathsf{PaPreOrd}(\mathbb C)$ of two morphisms $f \colon (A,\rho) \rightarrow (B,\sigma)$ and $g \colon (B,\sigma) \rightarrow (C,\tau)$ in $\mathsf{PaPreOrd}(\mathbb C)$ is defined by the external part of the following diagram
$$ \xymatrix@=15pt{ & & (A'', \rho'')  \ar[rd]^{f''} \ar@{>->}[dl]_{\alpha'} & & \\
& {(A',\rho')} \ar@{>->}[dl]_{\alpha} \ar[dr]^{f'}& &{\, \,(B',\sigma')  } \ar@{>->}[dl]_{\beta} \ar[dr]^{g'}& \\
(A,\rho) \ar@{.>}[rr]_f& & (B,\sigma) \ar@{.>}[rr]_g & &  (C,\tau) }$$
where the upper part is a pullback: in other words, $$g \circ f = (\beta, g') \circ (\alpha, f')= (\alpha \alpha', g' f''). $$
By the elementary properties of pullbacks one sees that this composition is associative, and the identity in $\mathsf{PaPreOrd}(\mathbb C)$ on a preorder $(A,\rho)$ is given by the arrow
$$\xymatrix{& (A,\rho) \ar@{=}[dl]_{1} \ar@{=}[dr]^{1} & \\ (A, \rho) \ar@{.>}[rr]_1& & (A,\rho) }
 $$
 The following observation will be useful:
 \begin{lemma}\label{Functor-I}
 There is a functor $I \colon \mathsf{PreOrd}(\mathbb C) \rightarrow \mathsf{PaPreOrd}(\mathbb C)$ which is the identity on objects and such that, for any $f \colon (A,\rho) \rightarrow (B,\sigma)$ in $\mathsf{PreOrd}(\mathbb C)$, $I(f) \colon (A,\rho) \rightarrow (B,\sigma)$ in $\mathsf{PaPreOrd}(\mathbb C)$ is defined by
 $$
 \xymatrix{& (A,\rho) \ar@{=}[dl]_{1} \ar[dr]^{f} & \\ (A, \rho) \ar@{.>}[rr]_{I(f)}& & (B,\sigma) }$$
 \end{lemma} 
\noindent { \bf Convention.}
 From now on, we shall often write $A$ instead of $(A,\rho)$ to denote an internal preorder, thus dropping the relation $\rho$. Furthermore, an arrow $\xymatrix{\ar@{>->}[r] &  }$ will always denote a clopen subobject, whereas an arrow $\xymatrix{A \ar[r] & B }$ will be a morphism of preorders.

 Note that the ``intuition'' here should be that a diagram
 $$
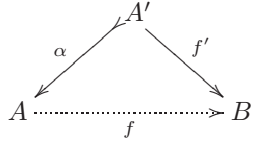
\xymatrix{& {\, \, \, \, A'} \ar@{>->}[dl]_{\alpha} \ar[dr]^{f'} & \\ A \ar@{.>}[rr]_f& & B}
 $$
``represents'' a morphism $f$ whose restriction on the clopen subobject $(A',\rho')$ of $(A,\rho)$ is $f'$, and that is ``trivial'' on the complement of $(A',\rho')$ in $(A,\rho)$.

 \begin{remark}
 The category $ \mathsf{PaPreOrd}(\mathbb C)$ is equipped with a natural ideal of morphisms $\mathcal N$ in the sense of Ehresmann \cite{Ehr}. Indeed, in a coherent category the initial object $0$ is strict: therefore, if we define $\mathcal N$ to be the class of morphisms in $\mathsf{PaPreOrd}(\mathbb C)$ of the form 
 $$
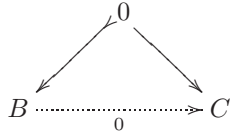
\xymatrix{& {\, \, 0 } \ar@{>->}[dl]_{} \ar[dr]^{} & \\ B \ar@{.>}[rr]_0& & C  }
 $$
 one easily sees that both $0 \circ f \colon (A,\rho) \rightarrow (C,\tau) $ and $g \circ 0 \colon (B,\sigma) \rightarrow (D,\psi)$ are again in $\mathcal N$ for any morphism $f \colon (A,\rho) \rightarrow (B,\sigma)$ and $g \colon (C,\tau) \rightarrow (D, \psi)$ in $ \mathsf{PaPreOrd}(\mathbb C)$.
 \end{remark}

 Starting from the category $\mathsf{PaPreOrd}(\mathbb C)$ of internal preorders and partial morphisms we are now going to define a quotient category of $\mathsf{PaPreOrd}(\mathbb C)$, that we shall call the \emph{stable category}. In the special case when $\mathbb C$ is the category of sets we shall extend the same construction as in the article \cite{FF} of Facchini and Finocchiaro, already presented in Section \ref{Section_stable_category} in the more general context of coherent categories.
 \begin{definition}
 A \emph{congruence diagram} in $\mathsf{PreOrd}(\mathbb C)$ is a diagram of the form
\begin{equation}\label{C-diagram}  
\xymatrix@=35pt{
{ {A_0^{1}}^c\,\, } \ar@{>->}[rr]^{{\alpha_0^1}^c} & & {\, \, A_1}  \ar@{>->}[ld]^{\alpha_1} \ar[dr]^{f_1}  & & \\
 {A_0 \,\, } \ar@{>->}[drr]_{\alpha_0^2} \ar@{>->}[rru]^{\alpha_0^1 } \ar@{>->}[r]_{\alpha_0} & A   & &B \\
 { {A_0^{2}}^c\,\, } \ar@{>->}[rr]_{{\alpha_0^2}^c}  & & A_2 \ar@{>->}[ul]_{\alpha_2} \ar[ru]_{f_2} & &
 }
\end{equation}
where:
 \begin{itemize}
 \item the two triangles commute;
 \item the clopen subobject $\xymatrix{{{A_0^{i}}^c\, \, }  \ar@{>->}[r]^{{\alpha_0^i}^c} & A_i }$ is the complement in $A_i$ of the clopen subobject $\xymatrix{{{A_0} \,\, } \ar@{>->}[r]^{{\alpha_0^i}} & A_i,}$ for $i=1,2$;
 \item  $f_1  \alpha_0^1 = f_2 \alpha_0^2$;
 \item each $f_i {\alpha_0^i}^c $ is a trivial morphism.
 \end{itemize}

 Given two parallel morphisms $(\alpha_1, f_1)$ and $(\alpha_2, f_2)$ in $\mathsf{PaPreOrd}(\mathbb C)$, depicted as
$$
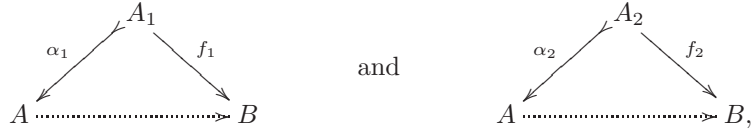
\xymatrix{
\ar@{}[drrrrrrr]|{\mbox{and}} & {\, \, \, \, A_1} \ar@{>->}[dl]_{\alpha_1} \ar[dr]^{f_1} & & & & & {\, \, \, \, A_2} \ar@{>->}[dl]_{\alpha_2} \ar[dr]^{f_2} & \\
 A \ar@{.>}[rr]_{} & & B & & & A \ar@{.>}[rr]_{}& & B,
 }
$$
 one says that they are \emph{equivalent}, and writes $(\alpha_1, f_1) \sim (\alpha_2, f_2)$, if there is a congruence diagram of the form \eqref{C-diagram}.
\end{definition}

 \begin{Prop}\label{cong} 
 The relation $\sim$ defined above is an equivalence relation which is also compatible with the composition in $\mathsf{PaPreOrd}(\mathbb C)$, and is then a \emph{congruence} on the category $\mathsf{PaPreOrd}(\mathbb C)$.
 \end{Prop}
 \begin{proof}
 It is clear that the relation $\sim$ is symmetric. To see that it is reflexive it suffices to choose $A_0=A_1$ in diagram \eqref{C-diagram}, so that ${A_0^{1}}^c=0$.
To see that $\sim$ is transitive consider three parallel morphisms in $\mathsf{PaPreOrd}(\mathbb C)$ such that $(\alpha_1, f_1) \sim (\alpha_2, f_2)$ and $(\alpha_2, f_2) \sim (\alpha_3, f_3)$. There are then a congruence diagram \eqref{C-diagram} and a congruence diagram 

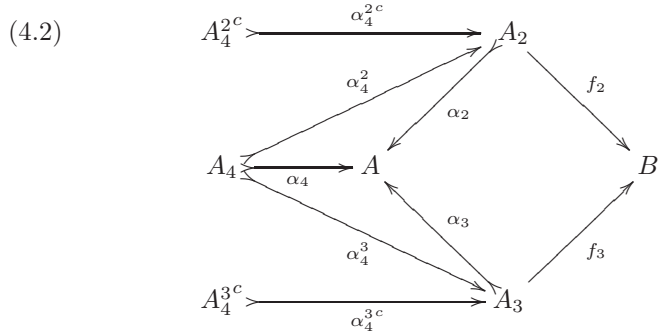
\begin{equation}\label{C2-diagram}  
\xymatrix@=35pt{{ {A_4^{2}}^c\,\, } \ar@{>->}[rr]^{{\alpha_4^2}^c} & & {\, \, A_2}  \ar@{>->}[ld]^{\alpha_2} \ar[dr]^{f_2}  & & \\
 {A_4 \,\, } \ar@{>->}[drr]_{\alpha_4^3} \ar@{>->}[rru]^{\alpha_4^2 } \ar@{>->}[r]_{\alpha_4} & A   & &B \\
 { {A_4^{3}}^c\,\, } \ar@{>->}[rr]_{{\alpha_4^3}^c}  & & A_3 \ar@{>->}[ul]_{\alpha_3} \ar[ru]_{f_3} & &
}
\end{equation}

This implies at once  that $f_1$, $f_2$ and $f_3$ coincide on $A_0\cap A_4$. It remains to prove that $f_1$ is trivial on the  complement of $A_0\cap A_4$ in $A_0$ and $f_3$ is trivial on the complement of $A_0\cap A_4$ in $A_3$. We treat the case of $f_1$: that of $f_3$ is analogous. First we know that $f_1$ is trivial on $A_0^{1^c }$, the complement of $A_0$ in $A_1$. It remains to prove that $f_1$ is trivial as well on the complement of $A_0\cap A_4$ in $A_0$.  But on $A_0$, $f_1$ and $f_2$ coincide; so it is equivalent to prove that $f_2$ is trivial on the complement of $A_0\cap A_4$ in $A_0$. Indeed, $f_2$ is trivial on $A_0^{2^c}$, the complement of $A_0$ in $A_2$, as well as on $A_4^{2^c}$, the complement of $A_4$ in $A_2$. Thus $f_2$ is trivial on the union of these two complements, which is the complement of $A_0\cap A_4$ in $A_2$. So $f_2$, and thus $f_1$, is trivial on the complement of $A_0\cap A_4$ in $A_0$, since $A_0$ is contained in $A_2$. This proves the transitivity.

Let us then show that the equivalence relation $\sim$ is compatible with the composition. Let us first prove that, given a morphism $$
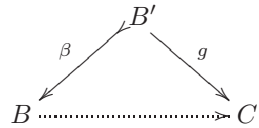
\xymatrix{& {\, \, \, \, B'} \ar@{>->}[dl]_{\beta} \ar[dr]^{g} & \\ B \ar@{.>}[rr]_{}& & C}
 $$
and two parallel morphisms $( \alpha_1, f_1) \colon A \rightarrow B$ and $( \alpha_2, f_2) \colon A \rightarrow B$ in $\mathsf{PaPreOrd}(\mathbb C)$ (as above) such that $( \alpha_1, f_1) \sim ( \alpha_2 , f_2)$, then $$ ( \beta, g) \circ ( \alpha_1, f_1) \sim ( \beta, g) \circ ( \alpha_2, f_2) .$$
By using the same notations as above for the congruence diagram \eqref{C-diagram} making $( \alpha_1, f_1)$ and $(\alpha_2 , f_2)$ equivalent, one can consider the composition diagram
$$
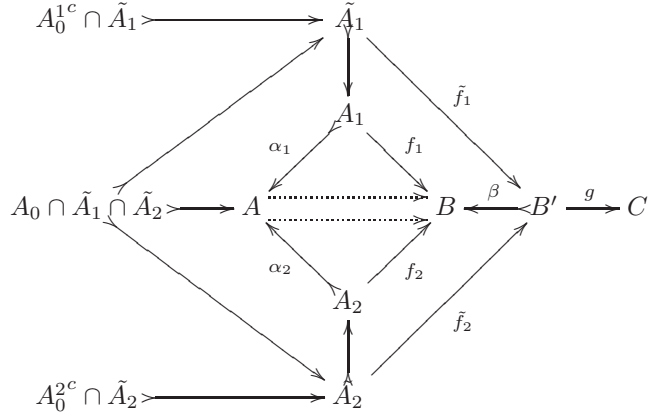
\xymatrix@=20pt{
{A_0^{1}}^c\cap \tilde{A_1} \, \,  \ar@{>->}[rr] & &{\, \tilde{A_1} }  \ar@{>->}[d] \ar[ddrr]^{\tilde{f_1}} & & & \\
& &  {\, A_1} \ar[dr]^{f_1} \ar@{>->}[dl]_{\alpha_1}  & & & \\
A_0 \cap \tilde{A_1} \cap \tilde{A_2} \, \, \ar@{>->}[r] \ar@{>->}[rruu]  \ar@{>->}[ddrr] & A \ar@<-1ex>@{.>}[rr] \ar@<1ex>@{.>}[rr] & &B & {\, B'} \ar[r]^g \ar@{>->}[l]_\beta& C \\
& & A_2 \ar[ru]_{f_2}  \ar@{>->}[lu]^{\alpha_2} & & & \\
{A_0^{2}}^c \cap \tilde{A_2} \, \,   \ar@{>->}[rr]  & & \tilde{A_2} \ar@{>->}[u] \ar[rruu]_{\tilde{f_2}} & & &
}
$$
where the two right-hand quadrangles are pullbacks by construction (so that $ \tilde{A_1} = {f_1}^{-1}(B')$ and $ \tilde{A_2} = {f_2}^{-1}(B')$). Note that the equality 
$f_1 \alpha_0^1 = f_2 \alpha_0^2$ implies that $$A_0 \cap \tilde{A_1}  = A_0 \cap \tilde{A_2} = A_0 \cap \tilde{A_1} \cap \tilde{A_2},$$ 
hence $f_1$ and $f_2$ coincide on $A_0 \cap \tilde{A_1} \cap {\tilde{A_2}}$. The definition of ${A_0^{1}}^c $ and the fact that the squares in the diagram
$$
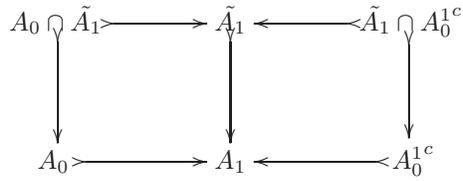
\xymatrix@=35pt{   {{A_0} \cap \tilde{A_1}  \,  }\ar@{>->}[r]  \ar@{>->}[d]   & \tilde{A_1} \ar@{>->}[d] & {\, \tilde{A_1} \cap  {A_0^{1}}^c } \ar@{>->}[l]  \ar@{>->}[d]    \\
{A_0 \, \, } \ar@{>->}[r] & A_1 & {\, \, {A_0^{1}}^c }\ar@{>->}[l]}
$$
are pullbacks (by construction) imply that the complement of ${A_0} \cap \tilde{A_1} $ in $\tilde{A_1}$ is $\tilde{A_1} \cap  {A_0^{1}}^c$ (since the pullback functor preserves complements). The restriction of $f_1$ to 
${A_0^{1}}^c$
is a trivial morphism, hence so is the restriction of $\tilde{f_1}$ to ${A_0^1}^c \cap \tilde{A _1}$. We then conclude that also the restriction of $g \tilde{f_1}$ to ${A_0^1}^c \cap \tilde{A _1}$ is trivial. Similarly one checks that the restriction of $g \tilde{f_2}$ to ${A_0^2}^c \cap \tilde{A _2}$ is trivial, so that one concludes that $ ( \beta, g) \circ ( \alpha_1, f_1) \sim ( \beta, g) \circ ( \alpha_2, f_2).$

Next consider 
$$
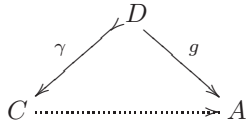
\xymatrix{& {\, \, \, \, D} \ar@{>->}[dl]_{\gamma} \ar[dr]^{g} & \\ C \ar@{.>}[rr]_{}& & A}
 $$
 and two morphisms $( \alpha_1, f_1) \colon A \rightarrow B$ and $( \alpha_2, f_2) \colon A \rightarrow B$ in $\mathsf{PaPreOrd}(\mathbb C)$ such that $( \alpha_1, f_1)  \sim  ( \alpha_2, f_2)$ with congruence diagram \eqref{C-diagram}.
 Consider the
 pullback
 $$
 \xymatrix{{C_0 \, \, }\ar@{>->}[r] \ar[d] & D \ar[d]^g \\
{ A_0\, \, } \ar@{>->}[r]_{\alpha_0} & A
 }
 $$ 
 and the diagram 
 $$
 \xymatrix@=40pt{  {{C_0^1}^c \, \, } \ar@{>->}[r] \ar[d]_{\overline{g}_1 } & {D_0^1\, \, } \ar@{>->}[r]^{\gamma_1} \ar[d]^{g_1} & D \ar[d]^g& {\, \, D_0^2} \ar@{>->}[l]_{\gamma_2}  \ar[d]^{g_2}& {\, \, {C_0^2}^c} \ar@{>->}[l] \ar[d]^{\overline{g}_2}\\ {{A_0^1}^c \, \, } \ar@{>->}[r] & {A_1\, \, }   \ar@{>->}[r]_{\alpha_1}& A &{\, \, A_2} \ar@{>->}[l]^{\alpha_2}   & {\, \, {A_0^2}^c} \ar@{>->}[l] 
 }
 $$
 where all the squares are pullbacks, ${A_0^i}^c$ is the complement of $C_0$ in $A_i$ and ${C_0^i}^c$ is the complement of $C_0$ in $D_0^i$ (for $i \in \{1,2\}$).
 It is then clear that $f_1 g_1$ and $f_2 g_2$ coincide on $C_0$ and, moreover, $f_i g_i$ are trivial morphisms when restricted to ${C_0^i}^c$ (since $f_i$ is trivial on ${A_0^i}^c$). It follows that 
 $$ ( \alpha_1, f_1) \circ ( \gamma, g)  \sim ( \alpha_2, f_2) \circ  ( \gamma, g),$$
 as desired.
 \end{proof}
 \begin{definition}
 We denote by $\mathsf{Stab}(\mathbb C)$ the quotient category of $\mathsf{PaPreOrd}(\mathbb C)$ by the congruence $\sim$ in Proposition \ref{cong}. If we write $\pi \colon \mathsf{PaPreOrd}(\mathbb C) \rightarrow \mathsf{Stab}(\mathbb C)$ for the quotient functor, we also have a functor $$\Sigma = \pi \circ I \colon \mathsf{PreOrd}(\mathbb C) \rightarrow \mathsf{Stab}(\mathbb C)$$ obtained by precomposing $\pi$ with the functor $I \colon \mathsf{PreOrd}(\mathbb C) \rightarrow \mathsf{PaPreOrd}(\mathbb C)$ in Lemma \ref{Functor-I}
 \end{definition}
 
{ \bf Convention.} From now on we shall use the notation $< \alpha, f > \colon A \rightarrow B$ for the morphism $\pi (\alpha, f) \colon A \rightarrow B$ in $\mathsf{Stab}(\mathbb C)$ (which is an equivalence class of morphisms in $\mathsf{PaPreOrd}(\mathbb C)$ by definition).
 
 \begin{lemma}\label{zero}
 The category $\mathsf{Stab}(\mathbb C)$ is pointed with zero object given by the initial object $0$ of $\mathsf{PreOrd}(\mathbb C)$.
 \end{lemma}
 \begin{proof}
 Given objects $A$ and $B$ the unique morphisms in $\mathsf{Stab}(\mathbb C)$ from $A$ to $0$ and from $0$ to $B$ are 
 $$\xymatrix{
\ar@{}[drrrrrrr]|{\mbox{and}} & {\, \,  0} \ar@{>->}[dl]_{} \ar@{=}[dr]^{} & & & & & { \, \, 0} \ar[dr]_{} \ar@{=}[dl]_{} & \\
 A \ar@{.>}[rr]_{\omega_A}& & 0 & & & 0 \ar@{.>}[rr]_{\alpha_B}& & B,
 }
 $$
 respectively. We shall keep these notations for these particular morphisms all throughout the paper. \end{proof}
 
 \begin{lemma}
 For an object $A \in \mathsf{PreOrd}(\mathbb C)$ the following conditions are equivalent:
 \begin{enumerate}
 \item the preorder $A$ is a trivial object;
 \item $\Sigma (A) \cong 0$.
 \end{enumerate}
 \end{lemma}
 \begin{proof}
$ (1) \Rightarrow (2)$ If $A$ is a trivial object, then we have the following congruence diagram:
$$
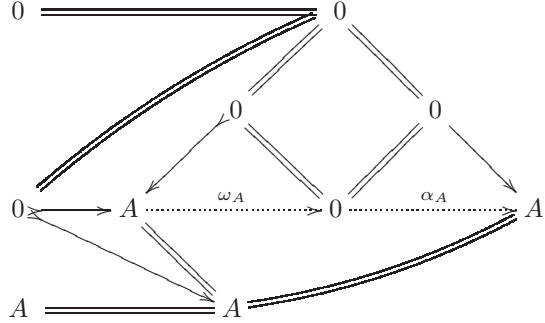
\xymatrix@=25pt{
{0 \,\, } \ar@{=}[rrr]& & & {\, \,0}  \ar@{=}[ld]^{} \ar@{=}[dr]^{}  & & & \\
&& {\,\, 0} \ar@{>->}[dl] \ar@{=}[dr]^{} &  & 0 \ar@{=}[dl] \ar[dr] & & \\
 {0 \,\, } \ar@{>->}[drr]_{} \ar@{=}@/^/@<0.5ex>[rruur]^{} \ar@{>->}[r]_{} & A   \ar@{.>}[rr]^{\omega_A} & & 0 \ar@{.>}[rr]^{\alpha_A}  && A\\
 { A \,\, } \ar@{=}[rr]_{}  & & A \ar@{=}[ul]_{}  \ar@{=}@/_/@<0.3ex>[urrr]_{} &  & & & 
 } $$
where the upper part of the diagram is the composite of the unique arrow  $\omega_ A \colon A \rightarrow 0$ followed by the unique arrow $\alpha_A \colon 0 \rightarrow A$ from Lemma \ref{zero}.
It follows that $\alpha_A \circ \omega_A = 1_A$ in $\mathsf{Stab}(\mathbb C)$. Since obviously $\omega_0 \circ \alpha_0 = 1_0$, it follows that $A \cong 0$ in $\mathsf{Stab}(\mathbb C)$.

Conversely, let us assume that $\alpha_A \circ \omega_A = 1_A$ in $\mathsf{Stab}(\mathbb C)$, so that we have the congruence diagram

$$
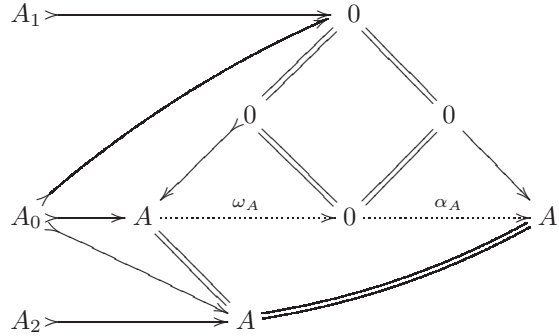
\xymatrix@=25pt{
{A_1 \,\, } \ar@{>->}[rrr]& & & {\, \,0}  \ar@{=}[ld]^{} \ar@{=}[dr]^{}  & & & \\
&&  {\,\, 0} \ar@{>->}[dl]  \ar@{=}[dr]^{} &  & 0 \ar@{=}[dl] \ar[dr] & & \\
 {A_0 \,\, } \ar@{>->}[drr]_{} \ar@{>->}@/^/@<0.5ex>[rruur]^{} \ar@{>->}[r]_{} & A   \ar@{.>}[rr]^{\omega_A} & & 0 \ar@{.>}[rr]^{\alpha_A}  && A\\
 { A_2 \,\, } \ar@{>->}[rr]_{}  & & A \ar@{=}[ul]_{}  \ar@{=}@/_/@<0.3ex>[urrr]_{} &  & & & 
 } $$
Since the initial object $0$ is strict, $A_0= 0= A_1$ and therefore $A_2 = A$, showing that $1_A$ is trivial in $\mathsf{PreOrd}(\mathbb C)$. There is then a factorization of $1_A$ through a trivial object, hence $A$ is a retract of a trivial object and is then itself trivial.
 \end{proof}
 \begin{Prop}\label{trivial-versus-zero}
 For a morphism $f \colon A \rightarrow B$ in $\mathsf{PreOrd}(\mathbb C)$ the following conditions are equivalent:
 \begin{enumerate}
 \item $f$ is a trivial morphism;
 \item $\Sigma (f)$ is a zero morphism in $\mathsf{Stab}(\mathbb C)$.
 \end{enumerate}
 \end{Prop}
 \begin{proof}
 If $f \colon A \rightarrow B$ is a trivial morphism in $\mathsf{PreOrd}(\mathbb C)$, it factors through a trivial object $D$. Since $\Sigma (D) \cong 0$ by the previous lemma, $\Sigma (f) \colon A \rightarrow B$ is a zero morphism.
 
 Conversely, if $\Sigma (f) \colon A \rightarrow B$ is a zero morphism in $\mathsf{Stab}(\mathbb C)$ we have a congruence diagram
$$
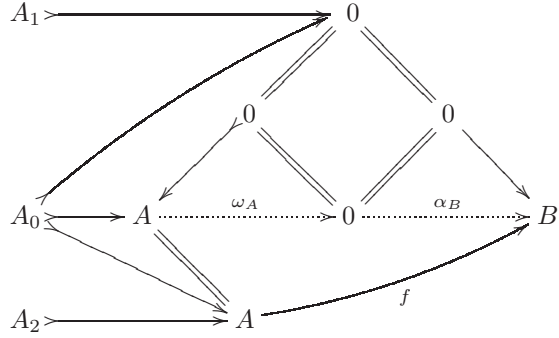
\xymatrix@=25pt{
{A_1 \,\, } \ar@{>->}[rrr]& & & {\, \,0}  \ar@{=}[ld]^{} \ar@{=}[dr]^{}  & & & \\
&&  {\,\, 0} \ar@{>->}[dl]  \ar@{=}[dr]^{} &  & 0 \ar@{=}[dl] \ar[dr] & & \\
 {A_0 \,\, } \ar@{>->}[drr]_{} \ar@{>->}@/^/@<0.5ex>[rruur]^{} \ar@{>->}[r]_{} & A   \ar@{.>}[rr]^{\omega_A} & & 0 \ar@{.>}[rr]^{\alpha_B}  && B\\
 { A_2 \,\, } \ar@{>->}[rr]_{}  & & A \ar@{=}[ul]_{}  \ar@/_/@<0.3ex>[urrr]_{f} &  & & & 
 } $$
 expressing the fact that $\Sigma (f)$ factors through $0$. We then have that $A_0=0=A_1$, and $A_2 = A$, hence $f$ is trivial on $A$, as desired.
 \end{proof}
 
 \begin{Prop}\label{zero-morphisms}
 For a morphism $\xymatrix@=30pt{A  \ar[r]^{< \alpha, f>} &  B }$ in $\mathsf{Stab}(\mathbb C)$ the following conditions are equivalent:
 \begin{enumerate}
 \item $< \alpha, f>= 0$
 \item $f$ is a trivial morphism in $\mathsf{PreOrd}(\mathbb C)$.
 \end{enumerate}
 \end{Prop}
 \begin{proof}
 If $< \alpha, f>= 0$ then there is a congruence diagram
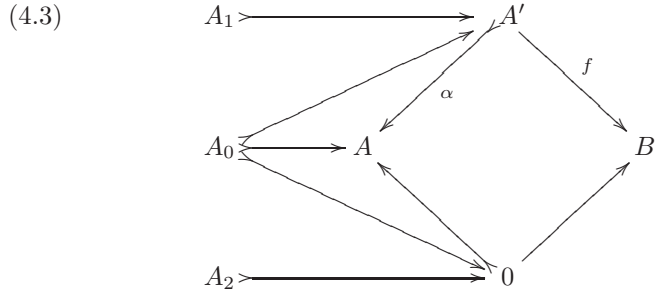
\begin{equation}\label{diagram-ref}
\xymatrix@=35pt{{ {A_1}\,\, } \ar@{>->}[rr] & & {\, \,\, \,  A'}  \ar@{>->}[ld]^{\alpha} \ar[dr]^{f}  & & \\
 {A_0 \,\, } \ar@{>->}[drr] \ar@{>->}[rru] \ar@{>->}[r] & A   & &B \\
 { A_2\,\, } \ar@{>->}[rr]  & & {\, \, 0 } \ar@{>->}[ul] \ar[ru] & &
 }
 \end{equation}
 hence $A_0 = 0 = A_2$ and then $A_1= A'$. This means that $f \colon A' \rightarrow B$ is a trivial morphism on $A'$. 
 
 Conversely, to see that the triviality of $f$ implies that $< \alpha, f>= 0$ it suffices to choose $A_1 = A'$ and then $A_2=0$ in the diagram \ref{diagram-ref}, which then becomes the required congruence diagram.
 \end{proof}

\begin{Prop}\label{Coincidence}
There is a commutative diagram
$$
\xymatrix{
\PreOrd C^* \ar[r]^J \ar[d]_{\mathcal{S}^*=\Sigma'} & \PreOrd C \ar[d]^{\Sigma}\\
\Stab ^* \ar[r]^{J'} & \mathsf{Stab}(\mathbb C)
}
$$
where $J$ and $J'$ are inclusion functor, $\Sigma'$ is the restriction of $\Sigma$ to $\PreOrd C ^*$ and ${\mathcal S}^*$ is the quotient functor defined in (\ref{quotient-category}). That is:
\begin{enumerate}
\item
Two parallel morphisms $f \colon A \rightarrow B$ and $g \colon A \rightarrow B$ in $\mathsf{PreOrd}(\mathbb C)$ are identified by the functor $\Sigma$ if and only if they belong to the equivalence relation $R_{A,B}$ defined at the beginning of Section \ref{Section_stable_category}.
\item
If $B$ is an object with a global element in $\mathsf{PreOrd}(\mathbb C)$ and $\langle \alpha,f\rangle\colon A\to B$ is a morphism in $\mathsf{Stab}(\mathbb C)$, then it is represented by a morphism in $\mathsf{PreOrd}(\mathbb C)$.
\end{enumerate}
 \end{Prop}

 \begin{proof}
   (1) The condition that $\Sigma (f) = \Sigma(g)$ means that
 there is a congruence diagram 
 $$
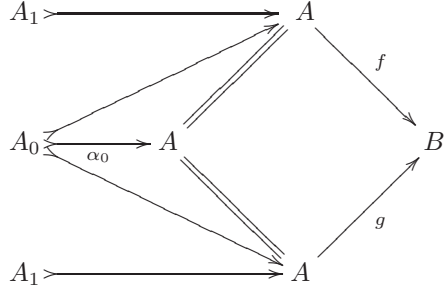
\xymatrix@=35pt{{ A_1 \,\, } \ar@{>->}[rr]^{} & & {\, \, A}  \ar@{=}[ld]^{} \ar[dr]^{f}  & & \\
 {A_0 \,\, } \ar@{>->}[drr]_{} \ar@{>->}[rru]^{ } \ar@{>->}[r]_{\alpha_0} & A   & &B \\
 { A_1\,\, } \ar@{>->}[rr]_{}  & & A \ar@{=}[ul]_{} \ar[ru]_{g} & &
 } 
$$
Accordingly, the parallel morphisms $f$ and $g$ are equal on a clopen subobject $A_0$ of $A$ and trivial on its complement $A_1$ in $A$, that is $(f,g) \in R_{A,B}$. 

 (2) Let us then assume that ${B \in \mathsf{PreOrd}(\mathbb C)^*}$, i.e. there is a morphism $b \colon 1 \rightarrow B$.
We still have to prove that any morphism $\langle \alpha,f\rangle \colon A\to B$ in $\mathsf{Stab}(\mathbb C)$ is the image of an arrow $g \colon A \rightarrow B$ in $\mathsf{PreOrd}(\mathbb C)$, i.e. that $\langle \alpha,f\rangle = \Sigma (g)$. Let us consider a morphism $(\alpha , f )$ in $\mathsf{PaPreOrd}(\mathbb C)$:
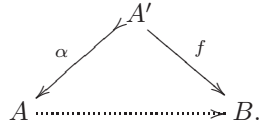
$$\xymatrix{& {\, \, \, \, A'} \ar@{>->}[dl]_{\alpha} \ar[dr]^{f} & \\ A \ar@{.>}[rr]. & & B.}
 $$
The clopen subobject $A'$ of $A$ induces a decomposition of $A$ into a coproduct $A = A' \amalg {A'}^c$, where ${A'}^c$ is the complement of $A'$ in $ A$, and then a commutative diagram
 $$
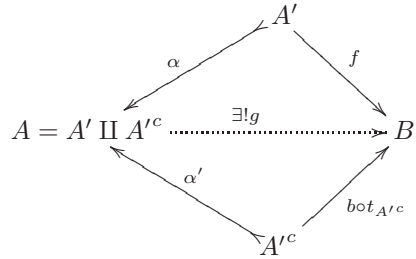
\xymatrix@=30pt{
  & & {\, \, \, A'}  \ar@{>->}[ld]_{\alpha} \ar[dr]^{f}  & & \\
  & A  = A' \amalg {A'}^c \ar@{.>}[rr]^{\exists ! g}&  &B \\
  & & {A'}^c \ar@{>->}[ul]_{\alpha'} \ar[ru]_{b \circ t_{{A'}^c}} & &
 } 
$$
 where $t_{A'^c}$ is the unique arrow $A'^c\to 1$ and $g$ is the unique morphism determined by the universal property of the coproduct making the triangles commute. To see that $(\alpha, f) \sim (1_A, g)$ it suffices to consider the congruence diagram
$$
\xymatrix@=30pt{{ 0 \,\, } \ar@{>->}[rr]^{} & & {\, \, \, \, A'}  \ar@{>->}[ld]^{\alpha } \ar[dr]^{f}  & & \\
 {A' \,\, } \ar@{>->}[drr]_{} \ar@{=}[rru]^{ } \ar@{>->}[r]_{\alpha} & A   & &B \\
 { {A'}^c\,\, } \ar@{>->}[rr]_{\alpha'}  & & A \ar@{=}[ul]_{} \ar[ru]_{g} & &
 } 
$$
where $g \circ \alpha' = b \circ t_{A'^c}$ is trivial since it factors through the trivial object $1$. 
 \end{proof}

\begin{remark} The approach we adopted in this section can be used to recover the construction of the stable category provided in \cite{FH} for the category of endomappings of a finite set, that is, the category $\Cal M$ whose objects are all pairs $(X,f)$, where $X=\{1,2,\dots,n\}$ for some $n\geq 0$ is a finite set and $f\colon X \to X$ is a mapping and a morphism $g\colon (X,f)\to (X',f')$ in $\Cal M$ is a mapping $g\colon X\to X'$ such that $f'g=gf$. The category $\mathcal{M}$ can be embedded in $\mathsf{PreOrd}(\mathsf{Set})$ via the assignment $(X,f)\mapsto (X,\rho_f)$, where $\rho_f$ is the preorder on $X$ defined by $x\rho_f y$ if and only if $x=f^t(y)$ for some integer $t\geq0$. In this way $\Cal M$ is identified with a (non-full) subcategory of $\mathsf{PreOrd}(\mathsf{Set})$, and clopen subobjects of a given object $(X,\rho_f)$ correspond to (unions of) the connected components of the graph associated with $f$.

A full comprehensive comparison between our approach and the one used in \cite{FH} would require to introduce some notions that we believe it is not convenient to include in this paper. This will be done in \cite{BCG2}, where the two constructions will be compared in detail.
\end{remark}

\section{The pretopos context}\label{Section_pretoposes}

Recall that a category $\Cal C$ with finite sums (=coproducts) is {\em extensive} if it has pullbacks along coprojections in a sum and the following condition holds: in the commutative diagram, where the bottom row is a sum
$$
\xymatrix{
X' \ar[r] \ar[d] & A \ar[d] & Y' \ar[l] \ar[d]\\
X \ar[r] & X \coprod Y & Y\ar[l]
}
$$
the top row is a sum if and only if the two squares are pullbacks \cite{CLW}.  The property saying that the upper row of the diagram is a sum whenever the two squares are pullbacks is usually referred to as the ``universality of sums". 
 Recall that a sum of two objects $A$ and $B$ is called \emph{disjoint} if the coprojections $\varepsilon_A\colon A\to A \coprod B$ and $\varepsilon_B\colon B\to A \coprod B$ are monomorphisms and the intersection $A\cap B$ is an initial object in $\Sub(A\coprod B)$.
If $\Cal C$ has finite sums and pullbacks along sum coprojections, extensivity is equivalent to the property of having disjoint and universal finite sums \cite[Proposition~2.14]{CLW}; if $\Cal C$ is also a coherent category, disjoint sums are universal.

A {\em pretopos} is an exact and extensive category or, equivalently, an exact coherent category with finite disjoint coproducts.
\vspace{3mm}

\noindent {\bf Convention.} From now on $\mathbb C$ will always denote a \emph{pretopos}.
\vspace{3mm}

Let $\langle r_1,r_2 \rangle \colon \rho \rightarrow A \times A$ be an internal preorder on $A \in \mathbb{C}$ and let $B$ be a complemented subobject of $A$. Let $(\varepsilon_B, \hat{\varepsilon}_B)\colon (B,\rho_B)\to (A,\rho)$ and $(\varepsilon_{B^c}, \hat{\varepsilon}_{B^c})\colon  (B^c,\rho_{B^c})\to (A,\rho)$ be the inclusions, and let $\langle t_1,t_2 \rangle \colon \rho_B\to B \times B$ and $\langle t'_1,t'_2 \rangle\colon \rho_{B^c}\to B^c \times B^c$ be the induced internal preorders on $B$ and $B^c$, respectively.

 If $B$ is a clopen subobject of $(A,\rho)$, then Corollary~\ref{remark_disjoint_clopen} and the extensivity of $\mathbb{C}$ imply that the four commutative squares
\begin{equation}\label{4squares}
\xymatrix@!=35pt{
\rho_B \ar[r]^{\hat{\varepsilon}_B}\ar[d]^{t_1}\ar@{}[rd]|{(I)}	& \rho \ar[d]^{r_1}\ar@{}[rd]|{(II)} & \rho_{B^c} \ar[l]_{\hat{\varepsilon}_{B^c}} \ar[d]^{t'_1} & \rho_B \ar[r]^{\hat{\varepsilon}_B}\ar[d]^{t_2}\ar@{}[rd]|{(III)}	& \rho \ar[d]^{r_2}\ar@{}[rd]|{(IV)} & \rho_{B^c} \ar[l]_{\hat{\varepsilon}_{B^c}} \ar[d]^{t'_2}\\
B \ar[r]^{\varepsilon_B} & A  	&	B^c  \ar[l]_{\varepsilon_{B^c}} &	B \ar[r]^{\varepsilon_B} & A  	&	B^c  \ar[l]_{\varepsilon_{B^c}} \\
}
\end{equation}
are all pullback diagrams.

Conversely, assume that both the squares $(I)$ and $(II)$ are pullbacks, and let us then show that $B$ is a clopen subobject of $(A,\rho)$.
Set $X:=(B\times B^c) \cap \rho$, where both $B\times B^c$ and $\rho$ are seen as subobjects of  $A\times A$, and consider the following commutative diagram
$$
\xymatrix{
\rho_B\ar@/_/[dddr]\ar@/^/[drr]\\
 & {\, X\, } \ar[r]\ar[d]\ar@{.>}[lu] & \rho\ar[d]\\
 & {B\times B^c\,\, } \ar[r]\ar[d] & A\times A\ar[d]\\
 & {B\,\,  }\ar[r]_{\varepsilon_B} & A,\\
}
$$
where the upper square is a pullback by definition and the vertical arrows of the bottom square are the projections of the first components. The dotted arrow exists because the outer quadrangle is a pullback by assumption, and it is clear that this arrow is then a monomorphism. Thus $X$ is a subobject of both $B \times B$ and $B\times B^c$ (all viewed as subobjects of $A \times A$), hence $X\cong 0$.
   Similarly, also $(B^c\times B)\cap \rho \cong 0$, and therefore $B$ is a clopen subobject of $(A,\rho)$. 

Moreover, the same conclusion also follows from the assumption that the squares $(I)$ and $(III)$ are pullbacks. To summarize, we have proved the following characterization of clopen subobjects in a pretopos.

\begin{Prop}\label{characterization_clopen}
Let $\langle r_1,r_2 \rangle \colon \rho \rightarrow A \times A$ be an internal preorder on an object $A$ in a pretopos $\mathbb{C}$, and let $B$ be a complemented subobject of $A$. Then the following conditions are equivalent:
\begin{enumerate}
\item
$B$ is a clopen subobject of $(A,\rho)$;
\item
the squares $(I)$ and $(II)$ in \ref{4squares} are both pullbacks;
\item
the squares $(I)$ and $(III)$ in \ref{4squares} are both pullbacks.
\end{enumerate}
\end{Prop}

\begin{remark}\label{remark_discrete_fibrations}
Asking that the square $(III)$ is a pullback can also be expressed as follows: when one looks at the morphism $(\varepsilon_B,\hat{\varepsilon}_B) \colon (B, \rho_B) \rightarrow (A, \rho)$ as an internal functor of internal categories, it is a \emph{discrete fibration}. 
Similarly, the requirement that the square $(I)$ is a pullback amounts to asking that this same morphism is a \emph{discrete opfibration}. When using this formulation of the definition of clopen subobject one sees a similarity with the notion of \emph{normal monomorphism} in the sense of Bourn \cite{Bourn}. The role of equivalence relations in the definition of normal monomorphisms is somehow similar to the one of preorders (in a pretopos) in the definition of clopens in this article.
\end{remark}

\begin{Prop}\label{lemma1_coproducts}
Let $(A,\rho)$ and $(B,\sigma)$ be two objects in $\PreOrd C$. Then the coproduct $(A,\rho)\coprod (B,\sigma)$ exists in $\PreOrd C$, and is defined ``componentwise'' by $(A\coprod B, \rho \coprod \sigma)$ as depicted in the following diagram
$$
\xymatrix@=40pt{
\rho \ar[r]^-{\hat{\varepsilon}_A} \ar@<.5ex>[d]^{r_2} \ar@<-.5ex>[d]_{r_1} & \rho \coprod \sigma \ar@<.5ex>[d]^{d_2} \ar@<-.5ex>[d]_{d_1} & \sigma \ar[l]_-{\hat{\varepsilon}_B} \ar@<.5ex>[d]^{s_2} \ar@<-.5ex>[d]_{s_1}\\
A \ar[r]^-{\varepsilon_A} & A\coprod B & B \ar[l]_-{\varepsilon_B} \\
}
$$
where $d_i: \rho\coprod \sigma \to A\coprod B$ is the unique arrow such that $d_i \hat{\varepsilon}_A=\varepsilon_A  r_i$ and $d_i \hat{\varepsilon}_B=\varepsilon_B s_i$ (for $i \in \{1,2\}$).\\
Moreover, $\rho = (\rho \coprod \sigma)_A$ and $\sigma = (\rho \coprod \sigma)_B$, that is, the orders $\rho$ on $A$ and $\sigma$ on $B$ coincide with those induced by $\rho \coprod \sigma$.
\end{Prop}

\begin{proof}
We first need to prove that $d:=\langle d_1,d_2\rangle \colon \rho \coprod \sigma\to (A\coprod B)\times (A\coprod B)$ is a monomorphism (hence a relation on $A\coprod B$). First notice that $\rho$ and $\sigma$ are disjoint subobjects of $(A\coprod B)\times (A\coprod B)$ (via the monomorphisms $(\varepsilon_A\times \varepsilon_A) \langle r_1,r_2 \rangle$ and $(\varepsilon_B\times \varepsilon_B) \langle s_1,s_2 \rangle$, respectively). Since $\rho \cup \sigma=\rho \coprod \sigma$, there is a monomorphism $d'\colon \rho \coprod \sigma \to (A\coprod B)\times (A\coprod B)$. Looking at the following diagram (where the vertical arrows of the bottom squares are the suitable projections) it is easy to see that $d=d'$.
$$
\xymatrix{
{\rho \, \, } \ar@{>->}[r] \ar[d] & \rho \coprod \sigma \ar[d]^{d'} &{\, \,  \sigma} \ar@{>->}[l]\ar[d]\\
{A\times A\, \, } \ar@{>->}[r] \ar[d] & (A\coprod B)\times (A\coprod B)\ar[d] & {\, \, B\times B} \ar@{>->}[l]\ar[d]\\
{A \, \, }\ar@{>->}[r]^-{\varepsilon_A} & A\coprod B & {\,\,  B}  \ar@{>->}[l]_-{\varepsilon_B}\\
}
$$

The reflexivity of $\rho$ and $\sigma$ implies that the relation $$d =\langle d_1,d_2\rangle \colon \rho \coprod \sigma\to (A\coprod B)\times (A\coprod B)$$ is reflexive. To check that this relation is also transitive, consider the pullback
$$
\xymatrix{
P\ar[r]^{p_2} \ar[d]_{p_1} & \rho \coprod \sigma \ar[d]^{d_1} \\
\rho\coprod \sigma \ar[r]_{d_2} & A\coprod B, \\
}
$$
and the following commutative diagram
$$
\xymatrix@=15pt{
\rho \times_ A \rho \ar[rd] \ar[dd] \ar@{.>}[rr]& & P \ar[rd] \ar[dd] & & \sigma \times_B \sigma \ar[rd] \ar[dd] \ar@{.>}[ll]& \\
 & \rho \ar[dd]_(0.4){r_1} \ar[rr] & & \rho \coprod \sigma \ar[dd]_(0.4){d_1} & & \sigma \ar[dd]_(0.4){s_1}\ar[ll]\\
\rho \ar[rd]^{r_2} \ar[rr] & & \rho \coprod \sigma \ar[rd]^{d_2} & & \sigma \ar[rd]^{s_2} \ar[ll]&\\
& A \ar[rr] &  & A\coprod B & & B, \ar[ll]\\ 
}
$$
where all the squares are pullbacks by the extensivity of $\mathbb{C}$. Again by extensivity, $P$ is the coproduct of $\rho\times_A \rho$ and $\sigma \times_B \sigma$. By using the transitivity of $\rho$ and $\sigma$, it is straightforward to show that there exists a morphism $\gamma: P \to \rho \coprod \sigma$ such that $d_1 \gamma=d_1 p_1$ and $d_2 \gamma=d_2 p_2$, yielding the desired transitivity. 

Finally, given an object $(X,{\tau})$ and morphisms $(A,\rho)\rightarrow ({X,\tau})\leftarrow (B,\sigma)$ in $\PreOrd C$ we have the following diagram
$$
\xymatrix@=15pt{
\rho \ar[rr] \paral{dd} \ar[rrrd]  & & \rho \coprod \sigma \paral{dd}^{d_2}_{d_1} \ar@{.>}[dr]^{\hat{\varphi}} & & \sigma \ar[ll] \paral{dd} \ar[dl]\\
 & & & \tau \paral{ddd}^{t_2}_{t_1}& \\
A \ar[rr]  \ar[rrrdd] & & A\coprod B \ar@{.>}[ddr]^{\varphi} & & B \ar[ll] \ar[ddl]\\
& & & & \\
 & & & X & \\
}
$$
where the (unique) dotted arrows induced by the coproducts at the two levels are such that $\varphi  d_i=t_i  \hat{\varphi}$, and also satisfy the required universal property.

For the last assertion, consider the diagram 
$$
\xymatrix{
\rho \ar@/_/[dddr]_{r_1}\ar@/^/[drr]^{\varepsilon_{\rho}}\\
 & {\, (\rho \coprod \sigma)_A \, } \ar[r]\ar[d]\ar@{.>}[lu] & \rho \coprod \sigma \ar[d]\\
 & {A\times A\,\, } \ar[r]_{\varepsilon_A \times \varepsilon_A} \ar[d] & B \times B \ar[d]\\
 & {A\,\,  }\ar[r]_{\varepsilon_A} & B,\\
}
$$
where the upper square is a pullback and the vertical arrows of the lower square are the first product projections.  The fact that the external diagram is a pullback (as observed above) implies that there is the induced dotted arrow showing that $(\rho \coprod \sigma)_A \le \rho$ as relations on $A$. Since one always has that $\rho \le (\rho \coprod \sigma)_A $, we conclude that $\rho = (\rho \coprod \sigma)_A$.
\end{proof}

\begin{corollary}\label{complemented-sub}
For a subobject $\xymatrix{(A,\rho)\, \, \,  \ar@{>->}[r]^j &\, (B, \sigma)}$ in $\mathsf{PreOrd}(\mathbb C)$ the following conditions are equivalent:
\begin{enumerate}
\item $\xymatrix{(A,\rho)\, \, \,  \ar@{>->}[r]^j &\, (B, \sigma)}$ is a clopen subobject;
\item $\xymatrix{(A,\rho)\, \, \,  \ar@{>->}[r]^j &\, (B, \sigma)}$ is a complemented subobject in $\mathsf{PreOrd}(\mathbb C)$. 
\end{enumerate}
\end{corollary}

\begin{proof}
$(1) \Rightarrow (2)$ By definition a clopen subobject $\xymatrix{(A,\rho)\, \, \,  \ar@{>->}[r]^j &\, (B, \sigma)}$ is such that $A$ is complemented in its codomain $B$, hence $ B \cong A \coprod A^c$. The assumption and Proposition \ref{lemma1_coproducts} imply that $\sigma = \rho \coprod \rho^c$.

$(2) \Rightarrow (1)$ Conversely, observe that in a pretopos $\mathbb C$, any two morphisms $f \colon A' \rightarrow A$ and $g \colon B' \rightarrow B$ induce a diagram
$$
\xymatrix@=35pt{A' \ar[r]^-{\varepsilon_{A'}} \ar[d]^f & A' \coprod B' \ar[d]^{f\coprod g} & B' \ar[d]_g  \ar[l]_-{\varepsilon_{B'}} \\
A \ar[r]_-{\varepsilon_A} & A \coprod B & \ar[l]^-{\varepsilon_B} B 
}
$$
where the two squares are pullbacks. Again by Proposition \ref{lemma1_coproducts} we know that this is also true in $\mathsf{PreOrd}(\mathbb C)$. Let us then consider the coproduct $$(A, \rho) \coprod (A^c,  \rho^c) = (A \coprod A^c, \rho \coprod \rho^c)$$ in $\mathsf{PreOrd}(\mathbb C)$, where $B = A \coprod A^c$, and the first projections 
yielding the diagram 
$$
\xymatrix@=35pt{\rho \ar[r]^-{\varepsilon_{\rho}} \ar[d]^{r_1} & \rho \coprod \rho^c \ar[d]^{r_1 \coprod s_1} & \rho^c \ar[d]_{s_1}  \ar[l]_-{\varepsilon_{\rho^c}} \\
A \ar[r]_-{\varepsilon_A} & A \coprod A^c & \ar[l]^-{\varepsilon_{A^c}} A^c. 
}
$$
By extensivity it follows that the two squares are pullbacks. Thanks to the characterization of clopen subobjects in Proposition \ref{characterization_clopen} and the last part of Proposition~\ref{lemma1_coproducts} we conclude that
$\xymatrix{(A,\rho)\, \, \,  \ar@{>->}[r]^j &\, (B, \sigma)}$ is a clopen subobject, as desired.
 \end{proof}

\begin{corollary}\label{pullbacks-sum}
Consider preorders $(A,\rho)$, $(B, \sigma)$, $(C, \tau)$ and $(D, \upsilon)$ in $\mathsf{PreOrd}(\mathbb C)$. Given any diagram 
of the form
$$
\xymatrix@=30pt{(A ,\rho)  \ar[r]^-{{\varepsilon}_A} \ar[d]^f & (A \coprod B, \rho \coprod \sigma) \ar[d]^{f \coprod g} & (B, \sigma)  \ar[l]_-{\varepsilon_B} \ar[d]_{g} \\
(C, \tau) \ar[r]_-{{\varepsilon}_C} & {(C \coprod D, \tau \coprod \upsilon)} & {(D, \upsilon).} \ar[l]^-{{\varepsilon}_D}
}
$$
in $\mathsf{PreOrd}(\mathbb C)$, the two commutative squares are pullbacks in $\mathsf{PreOrd}(\mathbb C)$.
\end{corollary}
\begin{proof}
The result holds in any pretopos, and it also holds in $\mathsf{PreOrd}(\mathbb C)$ thanks to Proposition \ref{lemma1_coproducts}, Proposition \ref{characterization_clopen} and Corollary \ref{complemented-sub}.
\end{proof}

\section{Preservation properties of $\Sigma$}\label{Preservation properties} 
 
 \begin{Prop}\label{Pres-Mono}
 The functor $\Sigma \colon \mathsf{PreOrd}(\mathbb C) \rightarrow \mathsf{Stab}(\mathbb C)$ preserves monomorphisms.
 \end{Prop}
 \begin{proof}
 Let $ m \colon A \rightarrow B$ be a monomorphism in $\mathsf{PreOrd}(\mathbb C)$, and consider two morphisms $< \alpha, f>$ and $< \beta, g>$ such that $\Sigma (m) \circ  < \alpha, f> = \Sigma (m) \circ  < \beta, g>$. We have a congruence diagram
$$
 
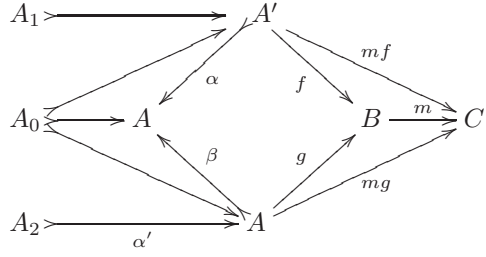
\xymatrix@=25pt{{ A_1 \,\, } \ar@{>->}[rr]^{} & & {\, \, \, \, A'}  \ar@{>->}[ld]^{\alpha } \ar[dr]_{f}  \ar[rrd]^{m  f} & & & \\
 {A_0 \,\, } \ar@{>->}[drr]_{} \ar@{>->}[rru]^{ } \ar@{>->}[r] & A   & &B \ar[r]^m & C\\
 { A_2\,\, } \ar@{>->}[rr]_{\alpha'}  & & A \ar@{>->}[ul]_{\beta} \ar[ru]^{g} \ar[rru]_{m g}& & &
 } 
 $$
 showing that $m  f$ and $m g$ coincide on $A_0$, and are both trivial on $A_1$ and $A_2$, respectively. Since $m$ is a monomorphism, we conclude that $f$ and $g$ coincide on $A_0$, and are trivial on $A_1$ and $A_2$, respectively (by Lemma \ref{epi-mono-trivial}). This implies that $< \alpha, f> = < \beta, g>$.
 \end{proof}
 
 \begin{Prop}\label{coproducts-preserved}
 The functor $\Sigma \colon \mathsf{PreOrd}(\mathbb C) \rightarrow \mathsf{Stab}(\mathbb C)$ preserves finite coproducts.
 \end{Prop}
 \begin{proof}
 We already know that $\Sigma$ preserves the zero object, then it will suffice to prove that $\Sigma$ also preserves binary coproducts.
 Consider then the coproduct $A \coprod B$ of two internal preorders $A$ and $B$ in $ \mathsf{PreOrd}(\mathbb C)$, and the diagram 
 $$
 
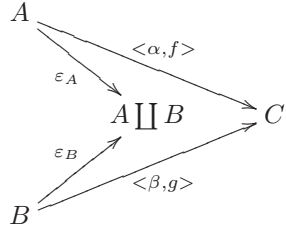
\xymatrix{
 A \ar[dr]_{\varepsilon_A} \ar[rrd]^{< \alpha, f>} & & \\
 &A \coprod B & C \\
B \ar[ur]^{\varepsilon_B} \ar[rru]_{<\beta , g>} & &  }
 $$
 We have to prove that there is a unique morphism $<\gamma, h> \colon A \coprod B \rightarrow C$ such that $ <\gamma, h>  \circ \varepsilon_A = < \alpha, f>$ and 
 $ <\gamma, h> \circ \varepsilon_B = < \beta, g>$ in $\mathsf{Stab}(\mathbb C)$.
 The following diagram in $\mathsf{PreOrd}(\mathbb C)$
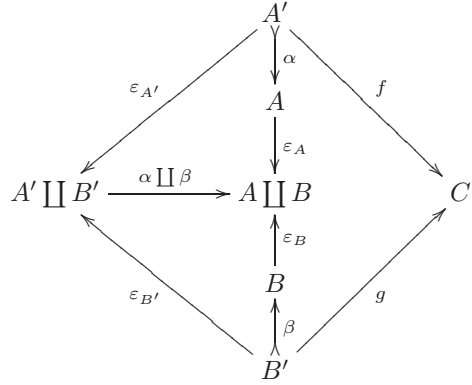
 $$
 \xymatrix@=20pt{
  & &  {A' } \ar@{ >->}[d]^{\alpha} \ar[ddrr]^f \ar[ddll]_{\varepsilon_{A'}}&  & & \\
    & &  A  \ar[d]^{\varepsilon_A} & & \\
 {A'} \coprod {B'} \ar[rr]^{\alpha \coprod \beta}  &
 &A \coprod B & & C \\
 & & B \ar[u]_{\varepsilon_B} & & & \\
 & & B' \ar@{ >->}[u]_{\beta} \ar[rruu]_g \ar[uull]^{\varepsilon_{B'}} & & &
  }$$
  induces a unique morphism $h \colon A' \coprod B' \rightarrow C$ such that $h \circ \varepsilon_{A'} = f$ and $h \circ \varepsilon_{B'} = g$. Observe that the complement ${A'}^c$ of $A'$ (in $A$) and ${B'}^c$ of $B'$ (in $B$) are such that ${A'}^c \coprod {B'}^c$ is the complement of $A' \coprod B'$ in $A \coprod B$. Since $A'$ and $B'$ are clopen subobjects in $A$ and $B$, respectively, then so is $A' \coprod B'$ in $A \coprod B$. We then get a morphism
  $$
  
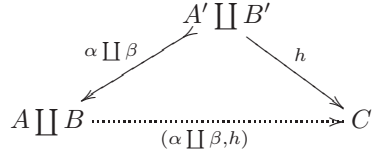
\xymatrix{
   & {\, \, \, \, A' \coprod B'} \ar@{>->}[dl]_{\alpha \coprod \beta} \ar[dr]^{h} & \\ A \coprod B \ar@{.>}[rr]_{(\alpha \coprod \beta, h)} & & C}
 $$
  in $\mathsf{Stab}(\mathbb C)$. Let us first check that
  $$<\alpha \coprod \beta, h> \circ \, \varepsilon_A = < \alpha, f>.$$
 This is clear by looking at the congruence diagram 
 $$
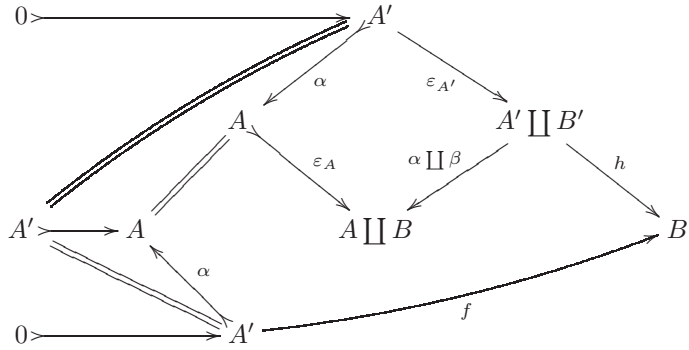
\xymatrix@=25pt{{0 \,\, } \ar@{>->}[rrr]& & & {\, \,A'}  \ar@{>->}[ld]^{\alpha} \ar[dr]_{\varepsilon_{A'}}  & & & \\
&&  { A\, \, } \ar@{=}[dl]  \ar@{>->}[dr]^{\varepsilon_A} &  & {A' \coprod B'} \ar[dl]_{\alpha \coprod \beta}  \ar[dr]^{h} & & \\
 {A' \,\, } \ar@{=}[drr]_{} \ar@{=}@/^/@<0.5ex>[rruur]^{} \ar@{>->}[r]_{} & A    & & A \coprod B   && B\\
 { 0 \,\, } \ar@{>->}[rr]_{}  & & A' \ar@{>->}[ul]_{\alpha}  \ar@/_/@<0.3ex>[urrr]_{f} &  & & & 
 } $$
 since $h  \varepsilon_{A'} = f$. Similarly, one checks that $<\alpha \coprod \beta, h> \circ \, \varepsilon_{B} = < \beta, g>$.
 To prove the uniqueness of the factorization one considers another morphism
  $$
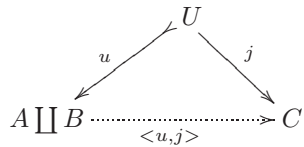
\xymatrix{& {\, \, \, \, U} \ar@{>->}[dl]_{u} \ar[dr]^{j} & \\ A \coprod B \ar@{.>}[rr]_{< u, j >}& & C}
 $$
 with the same properties as $< \alpha \coprod \beta, h>$, and the congruence diagram
 $$
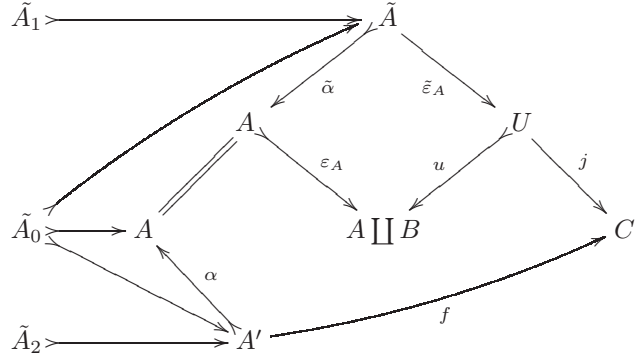
\xymatrix@=25pt{{\tilde{A_1} \,\, } \ar@{>->}[rrr]& & & {\, \,\tilde{A}}  \ar@{>->}[ld]^{\tilde{\alpha}} \ar[dr]_{\tilde{\varepsilon}_A}  & & & \\
&&  { A\, \, } \ar@{=}[dl]  \ar@{>->}[dr]^{\varepsilon_A} &  & {\, \, U} \ar@{>->}[dl]_{u}  \ar[dr]^{j} & & \\
 {\tilde{A_0} \,\, } \ar@{>->}[drr]_{} \ar@{>->}@/^/@<0.5ex>[rruur]^{} \ar@{>->}[r]_{} & A    & & A \coprod B   &&C\\
 { {\tilde{A}}_2 \,\, } \ar@{>->}[rr]_{}  & & A' \ar@{>->}[ul]_{\alpha}  \ar@/_/@<0.3ex>[urrr]_{f} &  & & & 
 } $$
 where the upper square is a pullback, $\tilde{A_1}$ is the complement of $\tilde{A}_0$ in $\tilde{A}$, and $\tilde{A_2}$ is the complement of $\tilde{A}_0$ in ${A'}$.
 Symmetrically, we also have the congruence diagram 
 $$
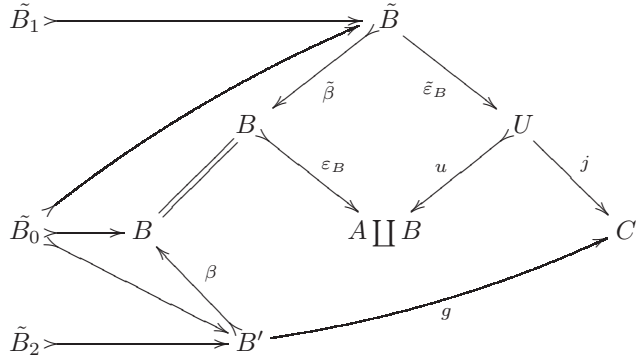
\xymatrix@=25pt{{\tilde{B_1} \,\, } \ar@{>->}[rrr]& & & {\, \,\tilde{B}}  \ar@{>->}[ld]^{\tilde{\beta}} \ar[dr]_{\tilde{\varepsilon}_B}  & & & \\
&&  { B\, \, } \ar@{=}[dl]  \ar@{>->}[dr]^{\varepsilon_B} &  & {\, \, U} \ar@{>->}[dl]_{u}  \ar[dr]^{j} & & \\
 {\tilde{B_0} \,\, } \ar@{>->}[drr]_{} \ar@{>->}@/^/@<0.5ex>[rruur]^{} \ar@{>->}[r]_{} & B   & & A \coprod B   &&C\\
 { {\tilde{B}}_2 \,\, } \ar@{>->}[rr]_{}  & & B' \ar@{>->}[ul]_{\beta}  \ar@/_/@<0.3ex>[urrr]_{g} &  & & & 
 } $$
To prove the uniqueness of the factorization we have to build a new congruence diagram of the form 
\begin{equation}\label{Congruence-h-unique}
\xymatrix@=30pt{{ V_2\, \, } \ar@{>->}[rr] & & {A' \coprod B'}  \ar@{>->}[ld]^{\alpha \coprod \beta} \ar[dr]^{h}  & & \\
 {V \, \, \, } \ar@{>->}[drr] \ar@{>->}[rru] \ar@{>->}[r] & A\coprod B   & &C \\
{ V_1\, \, } \ar@{>->}[rr]_{}  & & U \ar@{>->}[ul]_{u} \ar[ru]_{j} & &
 } 
 \end{equation}
 The universality of coproducts implies that $U= \tilde{A} \coprod \tilde{B}$ as clopen subobjects of $A\coprod B$, as we see in the following diagram where the two squares are pullbacks:
 $$
 \xymatrix{\tilde{A} \ar[r]^{\tilde{\varepsilon}_A} \ar[d]_{\tilde{\alpha}} & U \ar[d]^u & \tilde{B} \ar[l]_{\tilde{\varepsilon}_B} \ar[d]^{\tilde{\beta}} \\
 A \ar[r]_-{\varepsilon_A} & A \coprod B & B. \ar[l]^-{\varepsilon_B}
}
 $$
 In diagram \ref{Congruence-h-unique}, let us then choose $V = \tilde{A}_0 \coprod \tilde{B}_0$ and check that we do get a congruence diagram.
 First observe that $h$ and $f$ coincide on $A'$, therefore also on $\tilde{A}_0$. Moreover, $j$ and $f$ coincide on $\tilde{A}_0$, hence $h$ and $j$ coincide on $\tilde{A}_0$. Similarly, $h$ and $j$ coincide on $\tilde{B}_0$, and $h\mid_{V}=j\mid_V$. 
 On the other hand $h\mid_{A'}= f \mid_{A'}$, and $f$ is trivial on $\tilde{A}_2$ (which is a subobject of $A'$), hence $h$ is trivial on $\tilde{A}_2$. Similarly one sees that $h$ is trivial on $\tilde{B}_2$, and then $h$ is trivial on $\tilde{A}_2 \coprod \tilde{B}_2$. By observing that $\tilde{A}_2 \coprod \tilde{B}_2$ is the complement of $V = \tilde{A}_0 \coprod \tilde{B}_0$ in $ A' \coprod B'$, we set $V_2 = \tilde{A}_2 \coprod \tilde{B}_2$ in diagram \ref{Congruence-h-unique}. Similarly one sees that $j$ is trivial on $V_1 = \tilde{A}_1 \coprod \tilde{B}_1$, which is the complement of $V=\tilde{A}_0 \coprod \tilde{B}_0$ in $U = \tilde{A} \coprod \tilde{B}$.
  \end{proof}
 \begin{corollary}\label{coproducts}
 The stable category $\mathsf{Stab}(\mathbb C)$ has finite coproducts that are computed as in $\mathsf{PreOrd}(\mathbb C)$.
 \end{corollary}
 
 \section{Short exact sequences}\label{Short exact sequences}
 The aim of this section is to show that the functor $\Sigma \colon \mathsf{PreOrd}(\mathbb C) \rightarrow \mathsf{Stab}(\mathbb C)$ sends short $\mathcal Z$-exact sequences in $\mathsf{PreOrd}(\mathbb C)$ to short exact sequences in $\mathsf{Stab}(\mathbb C)$. We already observed that the $\mathcal Z$-kernel of any morphism in $\mathsf{PreOrd}(\mathbb C)$ exists (see Lemma \ref{Prop_prekernel}):
 
  \begin{Prop}\label{preservation-kernels}
The functor $\Sigma$ sends $\mathcal Z$-kernels in $\mathsf{PreOrd}(\mathbb C)$ to kernels in $\mathsf{Stab}(\mathbb C)$.
\end{Prop}
\begin{proof}
Let us consider a $\mathcal Z$-kernel $k \colon K \rightarrow A$ of a morphism $f \colon A \rightarrow B$ in $\mathsf{PreOrd}(\mathbb C)$. Since any $\mathcal Z$-kernel is a monomorphism, by Proposition \ref{Pres-Mono} we already know that $\Sigma (k)$ is a monomorphism. Given any morphism $<\mu, m> \colon M \rightarrow A$ in 
$\mathsf{Stab}(\mathbb C)$ such that $f \circ <\mu, m> = 0$ 
$$
\xymatrix{ & M \ar[dr]^{<\mu, m>} \ar@{.>}[dl]_{} & & \\
K  \ar[rr]_k& & A \ar[r]_f & B
}
$$
it will then suffice to prove the existence of a morphism $M \rightarrow K$ as above making the triangle commute. By Proposition \ref{zero-morphisms} we know that in the composition diagram 
$$ \xymatrix@=20pt{ & & M'  \ar[rd]^{m} \ar@{=}[dl] & & \\
& {\, \, \, \, M'}  \ar@{>->}[dl]_{\mu} \ar[dr]^{m}& &A \ar@{=}[dl] \ar[dr]^{f}& \\
M & & A  & &  B }$$
the arrow $f m$ is trivial (in $\mathsf{PreOrd}(\mathbb C)$). Since $k \colon K \rightarrow A$ is a $\mathcal Z$-kernel of $f$ this yields a (unique) factorization $g \colon M' \rightarrow K$ in $\mathsf{PreOrd}(\mathbb C)$ with $k  g = m$. Let us check that $< \mu, g> \colon M \rightarrow K$ is the required factorization in $\mathsf{Stab}(\mathbb C)$. Indeed the diagram
$$ \xymatrix@=20pt{ & & M'  \ar[rd]^{g} \ar@{=}[dl] & & \\
& {\, \, \, \, M'}  \ar@{>->}[dl]_{\mu} \ar[dr]^{g}& &K \ar@{=}[dl] \ar[dr]^{k}& \\
M & & K  & &  A }$$
shows that $k \circ < \mu, g> = < \mu, k  g > = <\mu, m>$, as desired.
\end{proof}
\begin{remark}
Let us mention here the fact that the stable category $\mathsf{Stab}(\mathbb C)$ has kernels, a fact that is not needed for this article, and that will be proved in \cite{BCG2}.
\end{remark}
Let us then consider the description of cokernels in the category $\mathsf{PreOrd}(\mathbb C)$. To avoid any ambiguity, let us make clear that when we refer to the transitive relation generated by a relation $\rho$, we mean -- when it exists -- the smallest transitive relation containing $\rho$.

\begin{Prop}\label{existence-cokernels}
Let $$\xymatrix{
(A, \rho) \ar[r]^{f} & (B, \sigma)  \ar[r]^q & (Q,\tau) 
}
$$
be two composable morphisms in $\mathsf{PreOrd} (\mathbb C)$. Then the following conditions are equivalent

\begin{enumerate}
\item $q \colon (B, \sigma) \rightarrow (Q,\tau)$ is the $\mathcal Z$-cokernel of $f \colon (A, \rho) \rightarrow (B, \sigma)$;
\item 
\begin{itemize}
\item[(a)]
the arrow $q$ is the coequalizer of $f  r_1$ and $f  r_2$ in $\mathbb C$
$$\xymatrix{\rho \ar@<-1ex>[r]_{r_2} \ar@<1ex>[r]^{r_1}  & A \ar[r]^f & B \ar@{->>}[r]^q & Q ;}$$ 
\item[(b)]
the transitive closure $\overline{U}$ of the relation $U= \sigma \cup f(\rho)^o$ exists in $B$;
\item[(c)]
$\tau  = q(\overline{U})$.
\end{itemize}
\end{enumerate}
\end{Prop}

\begin{proof}
$(1) \Rightarrow (2)$ \\
Let us first check condition $(a)$. If $q \colon (B, \sigma) \rightarrow (Q,\tau)$ is the $\mathcal Z$-cokernel of $f \colon (A, \rho) \rightarrow (B, \sigma)$, 
there is factorization $q'  f' \colon (A, \rho)  \rightarrow  (Z, \Delta_Z) \rightarrow  (Q, \tau)$ of $q f \colon (A, \rho) \rightarrow (Q, \tau)$ through a trivial object $(Z, \Delta_Z)$. This implies that $$q  f r_1 = q'   f'  r_1 = q'  1_Z  \hat{ f'} = q'  f'  r_2 = q  f r_2,$$
where $\hat{f'} \colon \rho \rightarrow Z$ makes the following diagram commute:
$$ \label{morphism}
\xymatrix{  \rho \ar@<.5ex>[d]^{r_2} \ar@<-.5ex>[d]_{r_1}  \ar[r]^{\hat{f'}}  & Z \ar@<.5ex>[d]^{1_Z} \ar@<-.5ex>[d]_{1_Z}  \\
 A \ar[r]_{f'} & {Z.} &
}
$$
Next consider the following diagram in $\mathbb C$
$$\xymatrix{\rho \ar@<-1ex>[r]_{r_2} \ar@<1ex>[r]^{r_1}  & A \ar[r]^f & B \ar[dr]_p \ar[r]^q & Q \\ & & & P}$$ 
where $p  f r_1 = p  f  r_2$ in $\mathbb C$.
We can define the indiscrete preorder structure $(P, P \times P)$ on $P$, so that the universal property of the $\mathcal Z$-cokernel $q$ induces an arrow $r$ in $\mathsf{PreOrd}(\mathbb C)$ such that $r q = p$:
$$\xymatrix{(A,\rho) \ar[r]_f  & (B, \sigma) \ar[dr]_p \ar[r]^q &  (Q, \tau)  \ar@{.>}[d]^r \\  & & (P, P \times P).}$$ 
The uniqueness of this factorization in $\mathbb C$ follows from the fact that any other factorization $r' \colon Q \rightarrow P$ in $\mathbb C$ induces an arrow $r' \colon (Q, \tau) \rightarrow (P, P \times P)$ in $\mathsf{PreOrd}(\mathbb C)$. It follows that $q$ is the coequalizer of $f r_1$ and $f  r_2$ in the category $\mathbb C$.\\
Observe then that $q(q^{-1}(\tau))  = \tau$, since $q$ is a regular epimorphism. To complete the proof of $(1) \Rightarrow (2)$ it will suffice to prove that $q^{-1}(\tau) = \overline{U}$, i.e. it is the smallest preorder on $B$ containing both $\sigma$ and $f(\rho^o)$.
First note that $q^{-1}(\tau)$ is a preorder containing $\sigma$, since $q$ is a morphism in $\mathsf{PreOrd} (\mathbb C)$. Moreover, $q^{-1}(\tau)$ contains $q^{-1}(\Delta_Q) = \mathsf{Eq}(q)$, since $\tau$ is reflexive. This implies that $f(\rho^o) \le \mathsf{Eq}(q) \le q^{-1}(\tau)$, which is a transitive relation since so is $\tau$.
Let $\alpha$ be any transitive relation on $B$ containing both $\sigma$ and $f(\rho)^o$ (note that $\alpha$ is in particular a preorder). The kernel pair $\mathsf{Eq}(q)$ of $q$ is the transitive relation generated by $f(\rho) \cup f(\rho)^o$, and this latter is clearly contained in $\sigma \cup f(\rho)^o \le \alpha$. This implies that
$$\mathsf{Eq}(q) \le \overline{\sigma \cup f(\rho)^o} \le \overline{\alpha} = \alpha.$$
Consider then the diagram
$$\xymatrix{(A,\rho) \ar[r]^f  & (B, \sigma) \ar[dr]_q \ar[r]^q &  (Q, \tau)  \ar@{.>}[d]^{1_Q} \\  & & (Q, q(\alpha))}$$ 
where $q(\alpha)$ is a preorder on $Q$ (since $\mathsf{Eq}(q) \le \alpha$), and $q f \colon (A, \rho) \rightarrow (Q, q(\alpha))$ is still a $\mathcal Z$-trivial morphism. The universal property of the $\mathcal Z$-cokernel $q$ yields a unique morphism in $\mathsf{PreOrd}(\mathbb C)$ making the right hand triangle commute, and this morphism has to be the identity on $Q$. This means that $\tau  \le q(\alpha)$, hence $q^{-1}(\tau)  \le q^{-1}( q(\alpha))$. But $\mathsf{Eq}(q) \le \alpha $, hence $q^{-1}( q(\alpha))= \alpha$ (see the proof of Lemma $1.3$ in \cite{FFG2}, for instance), and 
 $q^{-1}(\tau)  \le  \alpha$, as desired. \\

$(2) \Rightarrow (1)$ \\
Given $f \colon (A, \rho) \rightarrow (B, \sigma)$, we define $q$ as in $(2)(a)$, and $\tau = q(\overline{U})$ as in $(2)(c)$, where $U = \sigma \cup f(\rho)^o$ and its transitive closure $\overline{U}$ exists by $(2)(b)$.
By construction the morphism $q f \colon (A, \rho) \rightarrow (Q, \tau)$ is $\mathcal Z$-trivial. To prove the universal property of the $\mathcal Z$-cokernel consider the diagram $$\xymatrix{(A,\rho) \ar[r]^f  & (B, \sigma) \ar[dr]_p \ar[r]^q &  (Q, \tau)  \ar@{.>}[d]^{\pi} \\  & & (P, \beta)}$$
where $p$ has the property that $p f$ is a $\mathcal Z$-trivial morphism. The condition $2(a)$ implies that there is a unique morphism $\pi$ such that $\pi q = p$ in $\mathbb C$. It remains to prove that $\pi (\tau) \subset \beta$, since this will show that $\pi$ is a morphism in $\mathsf{PreOrd}(\mathbb C)$. Now, one obviously has that $U \le p^{-1}(\beta)$, and therefore $\overline{U} \le p^{-1}(\beta)$ since $p^{-1}(\beta)$ is transitive.
Accordingly,
$$\pi(\tau) = \pi (q(\overline{U})) = p(\overline{U}) \le p (p^{-1}(\beta)) \le \beta,$$
and the proof is complete.
\end{proof}
The following new notion will be useful for our main results:
\begin{definition}
A $\tau$-pretopos is a pretopos $\mathbb C$ with the property that the transitive closure of any relation on an object exists in $\mathbb C$.
\end{definition}

\begin{Prop}\label{existence-tau-cok}
When $\mathbb C$ is a $\tau$-pretopos the category $\mathsf{PreOrd}(\mathbb C)$ has all $\mathcal Z$-cokernels.
\end{Prop}
\begin{proof}
This follows from Proposition \ref{existence-cokernels}, provided we prove that the coequalizer in condition $(2)(a)$ exists. With the same notations as in that Proposition, let us consider the relation $f(\rho) \cup f(\rho)^o \cup \Delta_B$ on $B$ and its transitive closure $\overline{f(\rho) \cup f(\rho)^o \cup \Delta_B} = W$. It is easy to check that $W$ is the equivalence relation generated by $f(\rho) \cup f(\rho)^o \cup \Delta_B$: it suffices to observe that $W \cap W^o$ is a transitive relation containing $f(\rho) \cup f(\rho)^o \cup \Delta_B$, hence $W = W \cap W^o$. Since $W$ is also reflexive and symmetric, it is the smallest equivalence relation containing $f(\rho) \cup f(\rho)^o \cup \Delta_B$, and its coequalizer $B/W$ exists. 
\end{proof}
A $\sigma$-pretopos is one admitting denumerable unions of subobjects, that are preserved by pullbacks. It is well-known \cite{Elep} that any $\sigma$-pretopos is a $\tau$-pretopos, as is also any pretopos with arbitrary intersections of subobjects (in this case the transitive closure of a relation is obtained as the intersection of all the transitive relations containing it). In a $\sigma$-pretopos all finite colimits exist and are universal. In particular this is the case for the coequalizer of Proposition \ref{existence-cokernels}.

\begin{c-example}\label{HComp}
The category $\mathsf{HComp}$ of compact Hausdorff spaces is a $\tau$-pretopos that is not a $\sigma$-pretopos.
\end{c-example}
\begin{proof}
It is clear that the intersection of closed subspaces is closed, hence the intersection of subobjects is computed as in $\mathsf{Set}$. This implies that $\mathsf{HComp}$ is a $\tau$-pretopos. The same is true for the finite joins, but not for the denumerable ones. Indeed, consider for instance the space $X = [0, 2]$ in $\mathbb R$, $X_n = [1, 1 + \frac{1}{n}]$, for $n \in {\mathbb N}^*$. Observe that in $\mathsf{Set}$ one has $\cup_{n \in {\mathbb N}^*} X_n= ] 1,2]$, whereas $\vee_{n \in {\mathbb N}^*} X_n= [ 1,2]$ in $\mathsf{HComp}$. Observe then $[0,1] \cap [1,2] = \{1\}$, and $\bigvee_{n \in {\mathbb N}^*} [0, 1] \cap X_n= \emptyset,$ showing that  denumerable joins in $\mathsf{HComp}$ are not preserved under inverse images.
\end{proof}

\begin{Prop}
Any elementary topos is a $\tau$-pretopos.
\end{Prop}

\begin{proof}
Let $R\subseteq A\times A$ be a relation on the object $A$ in an elementary topos. The object $\Omega^{A\times A}$ of internal relations on $A$ is an internal locale and in particular, is provided with the internal intersection operation 
$\bigcap\colon \Omega^{\Omega^{A\times A}}\rightarrow \Omega^{A\times A}$
(see \cite{Elep}). Using the internal logic of the topos, one considers then the object of internal transitive relations on $A$ containing $R$
$$
T=\{S\mid (R\subseteq S)\wedge
(\forall a~\forall b~\forall c~(a,b)\in S \wedge (b,c)\in S \Rightarrow (a,c)\in S)\}
\subseteq \Omega^{A\times A}
$$
where $S$ is a variable of type $\Omega^{A\times A}$ and $a$, $b$, $c$ are variables of type $A$, while $R$ is a constant of type $\Omega^{A\times A}$.
The composite with $\bigcap$ of the corresponding global element of $\Omega^{\Omega^{A\times A}}$
$$
\bigcap \circ \ulcorner T \urcorner \colon 
\boldmath1 
\rightarrow
\Omega^{\Omega^{A\times A}}\rightarrow \Omega^{A\times A}
$$
yields a global element of $\Omega^{A\times A}$, that is, an actual subobject $\overline R\subseteq A\times A$. The routine set-theoretic proof showing that $\overline R$ is the smallest transitive relation on $A$ containing $R$ applies as such in the internal logic of the topos.
\end{proof}

\begin{remark}
When the topos has a Natural Number Object, $\overline R$ can also be defined as the object of those pairs of elements of $A$ which can be joined by a finite sequence of pairs in $R$.
\end{remark}

\begin{lemma}\label{pullback-clopen}
Let $\mathbb C$ be a pretopos, $f \colon A \rightarrow B$ a morphism in $\mathsf{PreOrd}(\mathbb C)$, and $\xymatrix{{B_1 \, \, } \ar@{>->}[r]^{\beta_1} &    B}$  a clopen subobject. If $B_2$ is the complement of $B_1$ in $B$, and the $\mathcal Z$-cokernels $q_i$ of $f_i \colon f^{-1} (B_i) \rightarrow B_i$ exist (for $i= 1, 2$), then we have the following diagram
$$\xymatrix@=35pt{f^{-1} (B_i) \ar[r]^{f_i} \ar@{>->}[d]_{\alpha_i} & B_i \ar@{>->}[d]_{\beta_i} \ar[r]^{q_i} & Q_i \ar@{>->}[d]^{\gamma_i} \\
A \ar[r]_f & B \ar[r]_{q_1 \coprod q_2}& Q 
}
$$
where the squares are pullbacks, the vertical morphisms are clopen subobjects, $q_i = {\mathcal Z}$-$\mathsf{coker}(f_i)$, $Q = Q_1 \coprod Q_2$, and $q_1 \coprod q_2 = {\mathcal Z}$-$\mathsf{coker}(f)$.
\end{lemma}
\begin{proof}
The left-hand squares are pullbacks by construction, and by corollaries \ref{complemented-sub} and \ref{pullbacks-sum} we know that the right-hand squares are pullbacks and each $\gamma_i$ is a clopen subobject. It remains to prove that $q_1 \coprod q_2$ is the $\mathcal Z$-cokernel of $f$. The universality of coproducts gives $A \cong f^{-1}(B_1) \coprod f^{-1}(B_2)$ thus $f = f_1 \coprod f_2$ and then 
$$
q f = (q_1 \coprod q_2)  (f_1 \coprod f_2) = q_1 f_1 \coprod q_2  f_2.
$$
is a trivial morphism as a coproduct of trivial morphisms. Let then $p \colon B \rightarrow C$ be such that $p  f$ is a trivial morphism. 
Then the equalities 
$$p  \beta_i  f_i = p  f  \alpha_i $$
imply that there are factorizations $r_i \colon Q_i \rightarrow C$ such that $r_i  q_i = p  \beta_i$, for $i \in \{1,2\}$. These morphisms $r_1, r_2$ induce a unique morphism $r \colon Q = Q_1 \coprod Q_2 \rightarrow C$ with $r \gamma_i = r_i$. Accordingly, 
$$r  q  \beta_i = r  \gamma_i  q_i = r_i q_i = p \beta_i, $$
and this implies that $r  q = p$. The uniqueness of the factorization then follows from the fact that $q_1$ and $q_2$ are epimorphisms, hence so is $q = q_1 \coprod q_2$.
\end{proof}
We leave the simple proof of the following result to the reader

\begin{lemma}
If $q \colon B \rightarrow Q$ is a $\mathcal Z$-cokernel in $\mathsf{PreOrd}(\mathbb C)$ and $h  q$ is a trivial morphism, then $h$ is a trivial morphism.
\end{lemma}

\begin{lemma}\label{existence-cokernel}
Let us consider a morphism $< \alpha, f>$ in $\mathsf{Stab}(\mathbb C)$ represented by 
$$\xymatrix{& {\, \, \, \, A'} \ar@{>->}[dl]_{\alpha} \ar[dr]^{f} & \\ A & & B}
 $$
and assume that for any clopen subobject $\xymatrix{{B' \, \, }\ar@{>->}[r] & B}$ the induced morphism $f^{-1} (B') \rightarrow B'$ has a $\mathcal Z$-cokernel in $\mathsf{PreOrd}(\mathbb C)$. Then the cokernel of $<\alpha, f>$ exists in $\mathsf{Stab}(\mathbb C)$, and 
$$\mathsf{coker} (<\alpha, f>) = \Sigma ({\mathcal Z}\mathrm{-}\mathsf{coker}(f)). $$
\end{lemma}
\begin{proof}
The assumption implies in particular that the $\mathcal Z$-cokernel of $f$ exists (it suffices to take $B'=B$). Let us then take the $\mathcal Z$-cokernel $q \colon B \rightarrow Q$ of $f$, and prove that $\Sigma (q)$ is the cokernel of $< \alpha, f>$ in $\mathsf{Stab}(\mathbb C)$.
One clearly has that $q \circ < \alpha ,  f> = <\alpha , q  f> = 0$ since $q  f$ is trivial (Proposition \ref{zero-morphisms}).
To check the universal property consider any morphism 
$$\xymatrix{& {\, \, \, B'} \ar@{>->}[dl]_{\beta} \ar[dr]^{g} & \\ B & & C}
 $$
 such that $<\beta, g> \circ <\alpha, f> = 0$.  
We have $<\beta, g> \circ <\alpha, f> = <\alpha \alpha' , g  f' >,$ where $\alpha '$ and $f'$ are defined by the following pullback:
$$
\xymatrix{{A'' \, \, } \ar@{>->}[r]^{\alpha '} \ar[d]_{f '} & {\, \,\,  A'} \ar[d]^{f}  \\
{B' \, \, }\ar@{>->}[r]_{\beta} & B.
}
$$
Let $q' \colon B' \rightarrow Q'$ be the $\mathcal Z$-cokernel of $f'$ in $\mathsf{PreOrd}(\mathbb C)$; the fact that $g f'$ is a trivial morphism implies that  there is a unique morphism $h \colon Q' \rightarrow C$ such that $h  q' = g$. Observe that Lemma \ref{pullback-clopen} implies that the square 
$$
\xymatrix{{B' \, \, } \ar@{>->}[r]^{\beta} \ar[d]_{q'} & {\, \,\,  B} \ar[d]^{q}  \\
{Q' \, \, }\ar@{>->}[r]_{\gamma} &{Q}
}
$$
is a pullback and $\gamma$ is a clopen subobject.
Then $< \gamma, h>$ is a morphism in $\mathsf{Stab}(\mathbb C)$, that is also the required factorization, since
$$<\gamma, h > \circ q = <\beta , h  q'> = <\beta, g>.$$
To prove the uniqueness, let $<\delta, s> \colon Q \rightarrow C$ be another morphism in $\mathsf{Stab}(\mathbb C)$ such that $<\delta, s> \circ q = <\beta, g>$.
This means that there is congruence diagram
\begin{equation}\label{c-diagram-uniqueness}
\xymatrix@=40pt{{B_1 \,\, } \ar@{>->}[rr]^{\beta_1}& &  {\, \,\, \, B''}  \ar@{>->}[ld]_{\beta '} \ar[r]^{q''}  & {Q''} \ar@{>->}[dl]_{\delta}   \ar[dr]^{s} & & \\
{B_0\,\, } \ar@{>->}@/^/@<0.3ex>[rru]^(.4){\beta_1'} \ar@{>->}[r]^{\beta_0} \ar@{>->}@/_/@<-0.3ex>[rrd]^(.4){\beta_2'} &  { B\, \, }   \ar[r]^{q} & Q &   & C & \\
 { B_2 \,\, } \ar@{>->}[rr]_{\beta_2}  & & B' \ar@{>->}[ul]_{\beta}  \ar@/_/@<0.3ex>[urr]_{g}   & & & 
 } \end{equation}
 where the upper quadrangle is a pullback.
 We can then form the commutative diagram
$$
 \xymatrix@=40pt{  {A_2\, \, } \ar@{>->}[r]^{\tilde{\alpha}_2} \ar[d]_{{f}_2 } & {A'' \, \, } \ar@{>->}[r]^{\alpha'} \ar[d]^{f'} & {A'} \ar[d]^f& {\, \,  \tilde{A} } \ar@{>->}[l]_{\tilde{\alpha}}  \ar[d]^{f''}& {\, \, A_1 } \ar@{>->}[l]_{\tilde{\alpha}_1} \ar[d]^{{f}_1}\\ 
 {B_0\, \, } \ar[d]_{q_2}\ar@{>->}[r]^{\beta_2'} & {B'\, \, }   \ar@{>->}[r]^{\beta}  \ar[d]^{q'} & B  \ar[d]^{q} &{\, \, B''} \ar@{>->}[l]_{\beta'}   \ar[d]^{q''} & {\, \, {B_0}} \ar[d]^{q_1}  \ar@{>->}[l]_{\beta_1'} \\
{ Q_2\, \, } \ar@{>->}[r]_{\gamma_2} & {Q'\, \,} \ar@{>->}[r]_{\gamma} & Q & {\, \, Q'' } \ar@{>->}[l]^{\delta} & {\, \, Q_1} \ar@{>->}[l]^{\delta_1}
 }
 $$
 where 
 \begin{itemize}
 \item the lower central left square is a pullback with $\beta$ and $\gamma$ clopen subobjects, $q = {\mathcal Z}$-$\mathsf{coker}(f)$ and $q' = {\mathcal Z}$-$\mathsf{coker}(f')$, as observed above;
 \item the upper central left square is a pullback with $\alpha'$ a clopen subobject (see above);
 \item the upper central right square is a pullback so that $\tilde{\alpha}$ is a clopen subobject;
 \item the lower central right square is a pullback with $\delta$ is a clopen subobject, so that $\beta'$ is a clopen subobject;
 \item by Proposition \ref{existence-cokernels}, $q$ is a regular epimorphism, so that $q''$ is also a regular epimorphism, hence $Q''$ is the regular image (in $\mathbb C$) of the composite $q \beta '$, with the order relation induced by the one on $Q$ (since $\delta$ is a clopen subobject). By comparing this to Lemma \ref{pullback-clopen} one concludes that $q'' =  {\mathcal Z}$-$\mathsf{coker}(f'')$;
 \item the upper right square is a pullback with $\beta_1'$ and hence $\tilde{\alpha}_1$ a clopen subobject;
 \item $q_1$ is defined as ${\mathcal Z}$-$\mathsf{coker}(f_1)$, and by Lemma \ref{pullback-clopen} the lower right square is then a pullback;
 \item the upper left square is a pullback by construction, hence $\tilde{\alpha}_2$ is a clopen subobject;
 \item by construction $q_2$ is the $\mathcal Z$-cokernel of $f_2$ and by Lemma \ref{pullback-clopen} the lower left square is a pullback.
 \end{itemize}
 Now, the two composites in the middle row are equal (they both represent the inclusion of $B_0$ into $B$), hence $q  \beta \beta_2' = q  \beta''  \beta_1'$. It follows that, in $\mathbb C$, they have the same regular image, so that $Q_1 \cong Q_2$. Both these preorders have the same preorder structure (since $\gamma \gamma_2$ and $\delta  \delta_1$ are clopen subobjects). It follows that $Q_1 \cong Q_2$ as objects in $\mathsf{PreOrd}(\mathbb C)$. There is then no restriction in assuming that $\delta \delta_1 = \gamma  \gamma_2$ and $q_1 =q_2$. Since each square in the diagram above is a pullback, one can also assume that $A_1= A_2$, $f_1= f_2$, yielding the following equalities:
$$
 s  \delta_1  q_1  =  s q''  \beta_1' =  g  \beta_2' =  h q'  \beta_2' 
 =  h  \gamma_2  q_2 = h  \gamma_2  q_1.
$$
 Since $q_1$ is an epimorphism, $s  \delta_1 = h  \gamma_2$, and we can form the following diagram
 $$
\xymatrix@=35pt{{ \overline{Q}_1' \,\, } \ar@{>->}[rr]^{{\overline{\gamma}'}} & & {\, \, Q'}  \ar@{>->}[ld]^{\gamma} \ar[dr]^{h}  & & \\
 {Q_1 =Q_2 \,\, } \ar@{>->}[drr]_{\delta_1} \ar@{>->}[rru]^{\gamma_2 } \ar@{>->}[r]_-{\delta  \delta_1 = \gamma  \gamma_2} &Q   & &C \\
 { \overline{Q}_1'' \,\, } \ar@{>->}[rr]_{\overline{\delta}' }  & & {Q''} \ar@{>->}[ul]_{\delta} \ar[ru]_{s} & &
 } $$
 where $\overline{Q}_1'$ is the complement of $Q_1$ in $Q'$ and $\overline{Q}_1''$ the complement of $Q_1$ in $Q''$. We have to prove that $h$ and $s$ are trivial on $\overline{Q}_1'$ and $\overline{Q}_1''$, respectively.
 We first consider $h$ and the pullbacks
 $$
 \xymatrix@=40pt{{B_0 \,\, }\ar@{>->}[r]^{\beta_2'} \ar[d]_{q_2} & B' \ar[d]^{q'} & {\, \, B_2} \ar[d]^{\overline{q}_2}  \ar@{>->}[l]_{\beta_2} \\
 {Q_2\, \, } \ar@{>->}[r]_{\gamma_2} &  Q' & {\, \, \overline{Q_1'}} \ar@{>->}[l]^{\overline{\gamma}'}
 }
 $$
 Now, looking at the diagram \ref{c-diagram-uniqueness} we see that $g$ is trivial on $B_2$. Then the equalities
 $$h {\overline{\gamma}'} \overline{q}_2 = h  q'  \beta_2 = g  \beta_2$$
 show that this morphism is trivial. Since $\overline{q}_2$ is a $\mathcal Z$-cokernel, we conclude that $h  {\overline{\gamma}'} $ is trivial.
 
 On the other hand, in the case of $s$ we consider the pullbacks
 $$
 \xymatrix@=40pt{{B_0 \,\, }\ar@{>->}[r]^{\beta_1'} \ar[d]_{q_1} & B'' \ar[d]^{q''} & {\, \, B_1} \ar[d]^{\overline{q}_1}  \ar@{>->}[l]_{\beta_1} \\
 {Q_1\, \, } \ar@{>->}[r]_{\delta_1} &  Q'' & {\, \, \overline{Q_1''}} \ar@{>->}[l]^{\overline{\delta}'}
 }
 $$
and the equalities
$$s  \overline{\delta}'   \overline{q}_1 = s q''  \beta_1$$
together with the assumption that $s q''$ is trivial on $B_1$ imply that $s  \overline{\delta}'$ is trivial, as desired.
\end{proof}
\begin{corollary}
Let $\mathbb C$ be a $\tau$-pretopos. Then $\mathsf{Stab}(\mathbb C)$ has all cokernels, and 
$$\mathsf{coker}(<\alpha, f>) = \Sigma ({\mathcal Z}\mathrm{-}\mathsf{coker}(f)).$$
\end{corollary}
\begin{proof}
This follows from Lemma \ref{existence-cokernel} and Proposition \ref{existence-tau-cok}.
\end{proof}
\begin{corollary}\label{preservation-cokernels}
Let $\mathbb C$ be a $\tau$-pretopos. Then the functor $\Sigma \colon \mathsf{PreOrd}(\mathbb C) \rightarrow \mathsf{Stab}(\mathbb C)$ sends $\mathcal Z$-cokernels to cokernels.
\end{corollary}
\begin{proof}
This follows from the definition of the functor $\Sigma$ and the previous Corollary.
\end{proof}
\begin{theorem}\label{main-theo}
Let $\mathbb C$ be a $\tau$-pretopos. Then the functor $\Sigma \colon \mathsf{PreOrd}(\mathbb C) \rightarrow \mathsf{Stab}(\mathbb C)$ sends short $\mathcal Z$-exact sequences to short exact sequences.
\end{theorem}
\begin{proof}
This follows immediately from Proposition \ref{preservation-kernels} and Corollary \ref{preservation-cokernels}.
\end{proof}

\end{document}